\documentclass[english]{article}
\usepackage{preamble}

\title{Rigid Algebras and Cospans}

\begin{document}
\maketitle

\begin{abstract}
  We introduce rigid algebras, a generalization of rigid categories to  arbitrary symmetric monoidal $(\infty,2)$-categories.
  We develop their general theory, showing in particular that the \emph{a priori} $(\infty,2)$-category of rigid algebras is in fact an $(\infty,1)$-category.
  For the $(\infty,2)$-category of cospans in an $(\infty,1)$-category $\cC$, we show that the $(\infty,1)$-category of rigid commutative algebras is canonically identified with $\cC$. 
  This identification is used to construct an adjunction between the cospan construction and the functor assigning to a symmetric monoidal $(\infty,2)$-category its $(\infty,1)$-category of rigid commutative algebras.
\end{abstract}

\begin{figure}[H]
  \centering
  \includegraphics[scale=0.18]{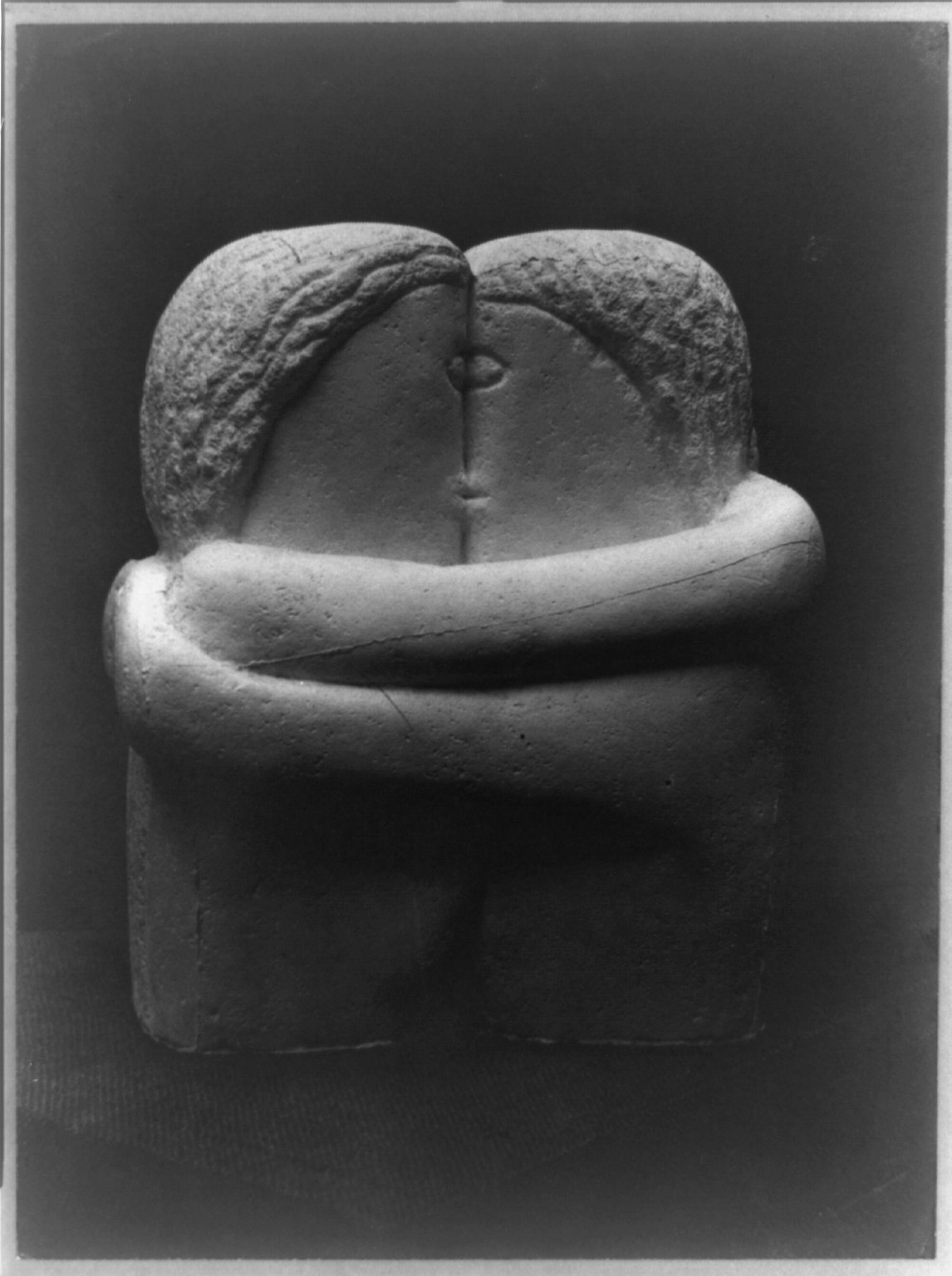}
  \captionsetup{labelformat=empty}
  \caption{The Kiss by Constantin Br\^{a}ncu\c{s}i}
\end{figure}

\pagebreak

\tableofcontents

\pagebreak

\section{Introduction}

Throughout the paper, we will refer to $(\infty,1)$-categories as $1$-categories, or simply categories. Similarly, $(\infty,2)$-categories will be called $2$-categories. When referring to ordinary categories and 2-categories, we specify them as $(1,1)$-categories and $(2,2)$-categories.

\subsection{Overview}

\subsubsection*{Rigid categories}

Rigid categories are a particularly well-behaved class of stable presentably monoidal categories.
Originally developed in the context of the geometric Langlands program \cite{gaitsgory2013sheaves,gaitsgory2019study,arinkin2022stacklocalsystemsrestricted}, and later studied in \cite{hoyois2021categorifiedgrothendieckriemannroch}, rigid categories have since become increasingly significant in condensed mathematics and continuous K-theory \cite{alex2025localizing}.
Among the rigid categories, those that are compactly generated are especially tractable: in such cases, rigidity is equivalent to the compact objects and dualizable objects coinciding.
Typical examples include the category of spectra $\Sp$, or more generally, the category of $R$-module spectra for a commutative ring spectrum $R$.
In contrast, for categories that are not compactly generated, compact objects do not adequately control the structure of the category.
In such cases, rigidity is instead defined in terms of the categorical properties of the monoidal structure, rather than through the behavior of a subcategory of compact objects.

\begin{definition}\label{d:rigid-category}
  A stable presentably monoidal category $\cA$ is called \emph{rigid} if:
  \begin{enumerate}
    \item the unit $\one\in \cA$ is compact, and
    \item the multiplication map $\mu:\cA\otimes \cA\to\cA$ is left adjoint as a map of $\cA$-bimodules, i.e., it is internally left adjoint in $\Prlst$ and a certain projection formula holds.
  \end{enumerate}
\end{definition}

An example of a rigid category that is not compactly generated is the category $\Shv(X)$ of sheaves of spectra on a compact Hausdorff space $X$.
Although rigid categories are not necessarily compactly generated, they are always \emph{compactly assembled}, meaning that they are dualizable objects in $\Prlst$.

The first condition in the above definition, that the unit $\one\in \cA$ is compact, can also be formulated in terms of adjunctions: it is equivalent to the unit map $\Sp\to \cA$ being internally left adjoint.
Using this formulation, the notion of rigid categories becomes readily generalizable. 
For instance, in \cite{ramzi2024locallyrigidinftycategories}, Ramzi develops the theory of rigidity for $\cB$-linear categories, where $\cB$ is any presentably symmetric monoidal category:

\begin{definition}
  A $\cB$-linear presentably monoidal category $\cA$ is called \emph{rigid} if:
  \begin{enumerate}
    \item the unit map $\eta\colon \cB\to \cA$ is internally left adjoint in $\Prl_\cB$, and
    \item the multiplication map $\mu:\cA\otimes_\cB \cA\to\cA$ is left adjoint as a map of $\cA$-bimodules.
  \end{enumerate}
\end{definition}
The definition in the setting of stable categories corresponds to the case where $\cB=\Sp$.

\subsubsection*{Rigid algebras}

This work starts with the observation that the definition of rigidity can be extended beyond the presentable setting to any symmetric monoidal 2-category $\cU$.

\begin{definition}
  An algebra $A$ in $\cU$ is called \emph{rigid} if:
  \begin{enumerate}
    \item the unit map $\eta \colon \one\to A$ is left adjoint in $\cU$, and
    \item the multiplication map $\mu \colon A\otimes A\to A$ is left adjoint as a map of $A$-bimodules.
  \end{enumerate}
\end{definition}

We denote by $\Algrig(\cU)\subseteq\Alg(\cU)$ the full sub-2-category of rigid algebras.
Like rigid categories, rigid algebras are always dualizable.
In fact, every rigid algebra carries a natural \emph{Frobenius algebra} structure, that connects its duality with its algebra structure.
Frobenius algebras have several equivalent definitions; we will present one that makes the relation to rigid algebras most transparent:

\begin{definition}[\cref{c:frob-equiv-explicit}]
  A \emph{Frobenius algebra} in a symmetric monoidal 1-category is an algebra $A$ equipped with:
  \begin{enumerate}
    \item A map $\epsilon\colon A\to \one$, called the \emph{counit}.
    \item A map $\delta\colon A\to A\otimes A$ of $A$-bimodules, called the \emph{comultiplication}.
    \item These satisfy counitality relations\footnote{For coherence, we should take only one of the relations as data, and merely require the other to exists.}:
     $$(\epsilon\otimes \id_A)\delta\simeq \id_A\simeq (\id_A\otimes \epsilon)\delta.$$
  \end{enumerate}
\end{definition}

The counit and comultiplication extend canonically to a full coalgebra structure.
The interaction between the algebra and coalgebra structures induces a self-duality, identifying a Frobenius algebra with its dual object.

In the case of a rigid algebra, we obtain a Frobenius structure by taking the right adjoints of the unit and multiplication as the counit and comultiplication respectively.
In this way, rigid algebras act as a categorification of Frobenius algebras: 
what is additional structure in the 1-categorical setting arises canonically through the existence of adjoints in the 2-categorical setting.
A (2,2)-categorical version of rigid algebras was studied by Walters and Wood \cite{Walters2008frobenius-object} under the name \emph{Frobenius objects}. Another related notion, also named \emph{rigid algebras}, was studied in \cite{brochier2021dualizabilitybraided} and further developed in \cite{decoppet2023rigidinfusion}.

We see that rigid algebras connect the symmetric monoidal notion of duality with the 2-categorical notion of adjunction.
This connection extends to maps of rigid algebras.
In general, given a map of dualizable objects $f\colon X\to Y$, one can construct a \emph{transpose} map between their duals $\trans{f}\colon \dual{Y}\to\dual{X}$.
If $A$ and $B$ are Frobenius algebras, and hence self-dual, then the transpose of a morphism $f\colon A\to B$ is a morphism $\trans{f}\colon B\to A$.

\begin{alphThm}[\cref{p:rigid-map-adjoint,c:rig-is-1-category}]\label{t:restatable-rigid-algebras}
  Let $A$ and $B$ be rigid algebras in $\cU$, then:
  \begin{enumerate}
    \item Every algebra map $f\colon A\to B$ is left adjoint to its transpose $\trans{f}\colon B\to A$.\label{i:adjoint-transpose}
    \item For any other algebra map $g\colon A\to B$, every 2-morphism of algebras $\alpha\colon f\Rightarrow g$ is invertible.\label{i:2-morphism-inv}
  \end{enumerate}
\end{alphThm}

Condition (\ref{i:2-morphism-inv}) implies that $\Algrig(\cU)$, although defined as a full sub-2-category of $\Alg(\cU)$, is in fact a \emph{1-category}.
The proof of \cref{t:restatable-rigid-algebras} is based on an analysis of rigid algebras in the lax arrow category (\cref{p:rigid-map-is-rigid}).

\subsubsection*{Cospans and rigid commutative algebras}

The most common examples of 2-categories are various ``categories-of-categories'', such as $\Prl_\cB$, and indeed it is in those contexts that rigid algebras were originally studied.
With our extended definition, however, rigid algebras can be considered in 2-categories of a broader variety.
In the second half of the paper, we turn our attention to the 2-category of (co)spans.

The 2-category of spans is known to satisfy a well-established universal property \cite{gaitsgory2019study,stefanich2020highersheaftheoryi,Macpherson2020ABY}: 
for any 1-category $\cC$ with finite limits, the 2-category $\Span(\cC)$ of spans in $\cC$ is initial among 2-categories under $\cC^\op$ satisfying a certain \emph{base change} property. 
Equipping $\cC$ with the cartesian symmetric monoidal structure, or equivalently $\cC^\op$ with the cocartesian structure, this enhances to a symmetric monoidal universal property.
In particular, as every $A\in \cC^\op$ has a unique commutative algebra structure, its image under $\cC^\op\to \Span(\cC)$ is endowed with a canonical commutative algebra structure.

We will use extensively the cocartesian structure on $\cC^\op$ and the resulting commutative algebra structures, so for ease of notation we replace $\cC$ with $\cC^\op$ and work with the 2-category of \emph{cospans}, which satisfies a corresponding universal property (\cref{p:universal-propertey-of-cospans}).
The main result of this paper establishes a new universal property for cospans, formulated in terms of rigid commutative algebras.
We show that the category of rigid commutative algebras in $\coSpan(\cC)$ is canonically identified with $\cC$ (\cref{p:can-equiv-rig}), and using this identification, we construct an adjunction between cospans and rigid commutative algebras.

\begin{alphThm}[\cref{t:cospan_cfrob_adj}]\label{t:restatable-rigid-cospan-adj}
  There is an adjunction 
  $$\coSpan\dashv \CAlgrig.$$
  Moreover, the unit of this adjunction is an isomorphism:
  $$\cC\isoto \CAlgrig(\coSpan(\cC)).$$
\end{alphThm}

Some care is required in formulating the above adjunction, since the two functors involved do not seem to share compatible sources and targets.
The cospan functor
$$\coSpan\colon \Cat_1^\rex\to \Catsm_2$$
takes 1-categories with finite colimits to symmetric monoidal 2-categories, while the rigid commutative algebras functor
$$\CAlgrig\colon \Catsm_2\to \Cat_1$$
returns 1-categories which do not necessarily have finite colimits.

To resolve this mismatch, we introduce a compatibility condition: we say that a symmetric monoidal 2-category 
$\cU$ is \emph{rigid bar compatible} if the category of rigid commutative algebras in 
$\cU$ admits pushouts.
Similarly, a symmetric monoidal 2-functor is rigid bar compatible if it preserves pushouts of rigid commutative algebras (see \cref{d:rigid-bar-compatible} for a precise statement).
This defines a subcategory $\Catrigbar\subseteq \Catsm_2$, and after restricting to this subcategory the sources and targets of the functors align:
$$\coSpan\colon \Cat_1^{\rex}\rightleftarrows\Catrigbar\colon \CAlgrig.$$

It is in this setting that we establish the adjunction.
In other words, we get an equivalence between right exact
functors $\cC\to \CAlgrig(\cU)$ and rigid bar compatible functors $\coSpan(\cC)\to \cU$.
In particular, we need to verify that the 2-category of cospans is rigid bar compatible: this follows from the isomorphism $\cC\isoto \CAlgrig(\coSpan(\cC))$ presented above.
We also note that $\Prl_\cB$, the primary context in which rigid algebras have previously been studied, is rigid bar compatible (\cref{e:prl-is-rigid-bar-compatible}).

\subsection{Organization}

We now describe the contents of each section of the paper.

We begin in \cref{s:projection-formula} by studying adjunctions of bimodule maps. 
Such adjunctions are characterized by a \emph{projection formula}, which we establish in the context of an arbitrary symmetric monoidal 2-category.
Here, bimodules are defined as symmetric monoidal functors from the envelope of the bimodule operad.
Thus, we will deduce the projection formula from a more general criterion for when a symmetric monoidal natural transformation admits a (lax) adjoint. 
To formulate this criterion, we rely on the theory of adjunctions between lax natural transformations as developed in \cite{abellan2024adjunctions}. 

In \cref{s:rigid-algebras}, we study the category of rigid algebras. 
We begin in the 1-categorical setting by reviewing the notions of duality and Frobenius algebras from \cite{ha}, with particular attention to duality of bimodules. We then pass to the 2-categorical setting, where we recall a characterization of duality in the lax arrow category by \cite{hoyois2017highertracesnoncommutativemotives}, and extend it to the setting of bimodules.
The main result of this section is that algebra maps between rigid algebras are themselves rigid algebras in the lax arrow category.
As a consequence, we show that such maps are adjoint to their transpose, and that every 2-morphism between rigid algebras is invertible.

Moving from the general theory to a specific example, in \cref{s:cospans} we study the 2-category of cospans.
After recalling the well-established universal property of cospans, we investigate the structure of rigid algebras in this 2-category.
The main result of this section is that the rigid commutative algebras in $\coSpan(\cC)$ are canonically identified with $\cC$, via the right-way functor.

In \cref{s:the-adjunction} we use this identification to formulate a new universal property for cospans, expressed as an adjunction between cospans and rigid commutative algebras.
To make this adjunction precise, we introduce the notion of rigid bar compatible categories.
The proof of the adjunction builds on the standard universal property of cospans.
In order to apply this property, we need to show that the forgetful functor from rigid commutative algebras satisfies \emph{cobase change}.
We prove that this holds in the setting of rigid bar compatible categories.

\subsection{Notations}

Throughout the paper we use the following notations:
\begin{itemize}
  \item We refer to $(\infty,1)$-categories as 1-categories, or simply categories. We usually denote 1-categories by $\cC$, $\cD$ or $\cE$. 
  \item $\Cat_1$ is the large 1-category of small 1-categories. The space of functors between $\cC,\cD\in \Cat_1$ is denoted $\Map(\cC,\cD)$ and the functor category is denoted $\Fun(\cC,\cD)$.
  \item $\Cat_1^\rex$ is the large 1-category of small 1-categories with finite colimits, and functors that commute with finite colimits (right exact functors). The space of right exact functors between $\cC,\cD\in \Cat_1^\rex$ is denoted $\Map^\rex(\cC,\cD)$.
  \item We refer to $(\infty,2)$-categories as 2-categories.
  We usually denote 2-categories by $\cU$, $\cV$ or $\cW$. 
  \item $\Cat_2$ is the large 1-category of small 2-categories. 
  The space of functors between $\cU,\cV\in \Cat_2$ is denoted $\Map(\cU,\cV)$ and the functor 2-category is denoted $\Fun(\cU,\cV)$.
  \item Given a 2-category $\cU$ and a morphism $f\colon X\to Y\in \cU$, we denote by $f^R$ its right adjoint if it exists. 
  The unit and counit of the adjunction are denoted $u_f\colon \id_X\Rightarrow f^Rf$ and $c_f\colon ff^R\Rightarrow \id_Y$ respectively.
  \item For a 1- (or 2-) category $\cC$, its \emph{core} is the maximal sub-groupoid $\cC^\simeq\subseteq \cC$. 
  For a 2-category $\cU$, its \emph{1-core} is the maximal sub-1-category $\cU^{\simeq_1}\subseteq \cU$.
  \item $\Catsm_n:= \CMon(\Cat_n)$ (for $n=1,2$) is the large category of small symmetric monoidal $n$-categories, which is formally given by commutative monoids in $\Cat_n$ with respect to the cartesian structure.
  The space of symmetric monoidal functors between $\cC,\cD\in \Cat_n^\otimes$ is denoted $\Mapo(\cC,\cD)$ and the symmetric monoidal functor (2-)category is denoted $\Funo(\cC,\cD)$.
  \item For $\cC\in \Catsm_n$, we denote the tensor product of $\cC$ by $\otimes$ and the unit of $\cC$ by $\one$.
  \item For $\cB$ a presentably symmetric monoidal category, $\Prl_\cB:=\Mod_\cB(\Prl)$ is the large symmetric monoidal 2-category of $\cB$-linear presentable categories. Sometimes in examples we will consider $\Prl_\cB$ as an element of $\Catsm_2$; to do that we implicitly move to a larger universe.
\end{itemize}

\subsection{Acknowledgments}
I would like to thank my advisor, Lior Yanovski, for his invaluable guidance and insights.
I would like to thank Achim Krause and Thomas Nikolaus for useful discussions.
I am also grateful to Shay Ben-Moshe, Dorin Boger, Shai Keidar and Shaul Ragimov for helpful conversations and suggestions.
Special thanks to Shay Ben-Moshe, Shai Keidar, Beckham Myers and Lior Yanovski for reviewing an earlier draft of this paper. 
Finally, I thank the anonymous referees for their careful reading and helpful comments. 

\section{The projection formula}\label{s:projection-formula}

Recall from \cite{ramzi2024locallyrigidinftycategories} that a $\cB$-linear presentably monoidal category $\cA\in \Alg(\Prl_\cB)$ is called \emph{rigid} if:
\begin{enumerate}
  \item the unit map $\eta \colon \cB\to \cA$ is internally left adjoint in $\Prl_\cB$, and
  \item the multiplication map $\mu \colon \cA\otimes_\cB \cA\to A$ is left adjoint as a map of $\cA$-bimodules.\label{i:rigid-prl-multiplication}
\end{enumerate} 
Condition (\ref{i:rigid-prl-multiplication}) can be checked using the \emph{projection formula}, which in this context appears in \cite{benmoshe2024naturalityofyoneda}.
For clarity, we will work in a more general setting: let $\cM,\cN\in  \Bmod[\cA]{\cA}(\Prl_B)$ and let $f\colon \cM\to \cN$ be a map of $\cA$-bimodules.
Denote by $\beta_\cM\colon \cA\otimes_\cB \cM\otimes_\cB \cA\to \cM$ the bi-action on $\cM$ (and similarly for $\cN$).
If $f$ has a right adjoint in  $\Prl_\cB$, then $f^R$ acquires the structure of a lax linear map of $\cA$-bimodules by the following composition: 
$$\beta_\cM(\id_\cA\otimes_\cB f^R\otimes_\cB \id_\cA)\xRightarrow{u_f}f^Rf\beta_\cM (\id_\cA\otimes_\cB f^R\otimes_\cB \id_\cA)\xRightarrow{\sim} f^R\beta_\cN (\id_\cA\otimes_\cB ff^R\otimes_\cB \id_\cA)\xRightarrow{c_f}f^R\beta_\cN.$$
The adjunction $f\dashv f^R$ is in $\Bmod[\cA]{\cA}(\Prl_\cB)$ if and only if the above composition is an isomorphism; in this case, we say that the projection formula holds.

To generalize the study of rigid algebras from $\Prl_\cB$ to an arbitrary symmetric monoidal 2-category $\cU$, we first generalize the projection formula to the setting of $\cU$.
In this context, bimodules are defined as symmetric monoidal functors from the envelope of the bimodule operad.
Accordingly, morphisms of bimodules are given by symmetric monoidal natural transformations. To understand adjunctions between such morphisms, we will adapt the methods developed by Abell\'an, Gagna, and Haugseng in \cite{abellan2024adjunctions}, which characterize adjunctions between natural transformations, to the symmetric monoidal setting.

\subsection{Lax natural transformations}

This subsection contains an overview of results from \cite{abellan2024adjunctions}.
Given 2-categories $\cU,\cV\in \Cat_2$ and functors $F,G\colon \cU\to \cV$, a \emph{lax natural transformation} $\alpha\colon F\Rightarrow G$ is given informally by maps $\alpha_X\colon FX\to GX$ for every $X\in \cU$ and lax commuting squares
\[\begin{tikzcd}
	FX & GX \\
	FY & GY
	\arrow["{\alpha_X}", from=1-1, to=1-2]
	\arrow["Ff"', from=1-1, to=2-1]
	\arrow["{\alpha_f}"', shorten <=6pt, shorten >=6pt, Rightarrow, from=1-2, to=2-1]
	\arrow["Gf", from=1-2, to=2-2]
	\arrow["{\alpha_Y}"', from=2-1, to=2-2]
\end{tikzcd}\]
for every $f\colon X\to Y$ (together with coherences).
The dual notion of an oplax natural transformation has the 2-morphism $\alpha_f$ going in the opposite direction.
Formally, both notions are defined via the \emph{Gray product} $\xlax$, a lax variant of the cartesian product on $\Cat_2$.
Specifically, the Gray product of two arrows, $\Delta^1\xlax \Delta^1$, yields the walking lax commuting square:
\[\begin{tikzcd}
	\bullet & \bullet \\
	\bullet & \bullet.
	\arrow[from=1-1, to=1-2]
	\arrow[from=1-1, to=2-1]
	\arrow[shorten <=8pt, shorten >=8pt, Rightarrow, from=1-2, to=2-1]
	\arrow[from=1-2, to=2-2]
	\arrow[from=2-1, to=2-2]
\end{tikzcd}\]
The Gray product was first introduced by Gray \cite{gray1974adjointness} in the context of (2,2)-categories, and various extensions to $(\infty,2)$-categories have since been proposed; see \cite{campion2023modelindependentgraytensorproduct} for a comparison of these constructions.
In this paper, we follow the approach of Gagna, Harpaz and Lanari \cite{gagna2021graytensorproducts}, which is further developed in \cite{haugseng2023twovariablefibrationsfactorisationsystems,abellan2024adjunctions}.
We will also reference the construction described in \cite{johnsonfreyd2017laxnatural}, as used in \cite{hoyois2017highertracesnoncommutativemotives}.

\begin{definition}
  Let $\cU,\cV\in\Cat_2$. 
  The 2-categories of functors with lax natural transformations $\Funlax(\cU,\cV)$ and oplax natural transformations $\Funoplax(\cU,\cV)$ are defined as the right adjoints
  \begin{align*}
    \Map(\cW,\Funlax(\cU,\cV))&\simeq \Map(\cW\xlax\cU,\cV)\\
    \Map(\cW,\Funoplax(\cU,\cV))&\simeq \Map(\cU\xlax\cW,\cV)
  \end{align*}
  for every $\cW\in \Cat_2$.
\end{definition}

To illustrate the data in $\Funlax(\cU,\cV)$, we will examine the lax arrow category $\Funlax(\Delta^1,\cV)$.

\begin{example}\label{e:lax-arrow-category}
  The low dimensional cells in $\Funlax(\Delta^1,\cV)$ are as follows:
  \begin{itemize}
    \item The objects are arrows $f\colon X\to Y\in \cV$.
    \item The 1-morphisms are lax squares
    \[\begin{tikzcd}
      X & X' \\
      Y & Y'.
      \arrow["", from=1-1, to=1-2]
      \arrow["f"', from=1-1, to=2-1]
      \arrow[shorten <=6pt, shorten >=6pt, Rightarrow, from=1-2, to=2-1]
      \arrow["f'", from=1-2, to=2-2]
      \arrow[from=2-1, to=2-2]
    \end{tikzcd}\]
    \item The 2-morpshims are commuting diagrams
      \[\begin{tikzcd}
        &&& {X'} \\
        X && {} \\
        & {} && {Y'.} \\
        Y
        \arrow["{{f'}}", from=1-4, to=3-4]
        \arrow[""{name=0, anchor=center, inner sep=0}, curve={height=-18pt}, from=2-1, to=1-4]
        \arrow[""{name=1, anchor=center, inner sep=0}, curve={height=18pt}, from=2-1, to=1-4]
        \arrow["f"', from=2-1, to=4-1]
        \arrow[shift right=2, curve={height=5pt}, shorten <=9pt, shorten >=5pt, Rightarrow, from=2-3, to=3-2]
        \arrow[shift left=2, curve={height=-5pt}, shorten <=5pt, shorten >=9pt, Rightarrow, from=2-3, to=3-2]
        \arrow[""{name=2, anchor=center, inner sep=0}, curve={height=-18pt}, from=3-4, to=4-1]
        \arrow[""{name=3, anchor=center, inner sep=0}, curve={height=18pt}, from=3-4, to=4-1]
        \arrow[shorten <=5pt, shorten >=5pt, Rightarrow, from=1, to=0]
        \arrow[shorten <=5pt, shorten >=5pt, Rightarrow, from=2, to=3]
      \end{tikzcd}\]
  \end{itemize}
\end{example}

Consider also the 2-category $\Fun(\cU,\cV)$ of functors with strong natural transformations.
There exist natural functors
$$\Funoplax(\cU,\cV)\leftarrow\Fun(\cU,\cV)\to \Funlax(\cU,\cV)$$
given by considering a strong natural transformation as (op)lax.
These functors are \emph{locally full inclusions} \cite[Corollary 2.8.12]{abellan2024adjunctions}, in the sense of the following definition:
\begin{definition}
  A functor of 2-categories $F\colon\cU\to \cV$ is called:
  \begin{enumerate}
    \item \emph{locally fully faithful} if $\hom_\cU(X,Y)\to \hom_\cV(FX,FY)$ is fully faithful for all $X,Y\in \cU$, and
    \item \emph{a locally full inclusion} if it is locally fully faithful and $\cU\core\to \cV\core$ is a monomorphism of spaces.
  \end{enumerate}
\end{definition}

According to \cite[Corollary 2.6.7]{abellan2024adjunctions}, $F\colon \cU\to \cV$ is a locally full inclusion if and only if it is right orthogonal to the maps $\pt\sqcup \pt\to \pt$ and $\partial C_2\to C_2$, with the last being the boundary inclusion of the walking 2-morphism.
In particular, a locally full inclusion is a monomorphism in $\Cat_2$ \cite[Proposition 2.6.10]{abellan2024adjunctions}. 

\begin{remark}
  The locally full inclusion $\Fun(\cU,\cV)\to \Funbilax(\cU,\cV)$ is also wide, which means that $\Fun(\cU,\cV)\core\to \Funbilax(\cU,\cV)\core$ is an isomorphism. 
  In fact, both sides are isomorphic to the mapping space $\Map(\cU,\cV)$.
\end{remark}

Let $F,G\colon \cU\to \cV$ be functors of 2-categories, and suppose that $\alpha\colon F\Rightarrow G$ is a strong natural transformation such that for each $X\in \cU$, $\alpha_X\colon FX\to GX$ has a right adjoint $\alpha_X^R\colon GX\to FX$.
Denote the unit and counit of the adjunction by $u_X\colon \id_{FX}\Rightarrow \alpha_X^R\alpha_X$ and $c_X\colon \alpha_X\alpha_X^R\Rightarrow \id_{GX}$.
Then for every $f\colon X\to Y$, the commuting naturality square
\[\begin{tikzcd}
	FX & GX \\
	FY & GY
	\arrow["{{\alpha_X}}", from=1-1, to=1-2]
	\arrow["Ff"', from=1-1, to=2-1]
	\arrow["Gf", from=1-2, to=2-2]
	\arrow["{{\alpha_Y}}"', from=2-1, to=2-2]
\end{tikzcd}\]
induces a lax commuting square
\[\begin{tikzcd}
	GX & FX \\
	GY & FY
	\arrow["{{\alpha_X^R}}", from=1-1, to=1-2]
	\arrow["Gf"', from=1-1, to=2-1]
	\arrow[shorten <=6pt, shorten >=6pt, Rightarrow, from=1-2, to=2-1]
	\arrow["Ff", from=1-2, to=2-2]
	\arrow["{{\alpha_Y^R}}"', from=2-1, to=2-2]
\end{tikzcd}\]
where the 2-morphism is the \emph{Beck-Chevalley} 2-morphism
$$Ff\alpha_X^R\xRightarrow{u_Y}
\alpha_Y^R\alpha_YFf\alpha_X^R\xRightarrow{\sim} 
\alpha_Y^RGf\alpha_X\alpha_X^R\xRightarrow{c_X}
\alpha_Y^RGf.$$
The following proposition implies that $\alpha_X^R$ assembles into a lax natural transformation $\alpha^R\colon G\Rightarrow F$, which is right adjoint to $\alpha$ in $\Funlax(\cU,\cV)$:

\begin{proposition}[{\cite[Corollary 5.2.10]{abellan2024adjunctions}}]\label{p:adj-lax-nat}
  Let $F,G\colon \cU\to \cV$ be functors of 2-categories.
  A lax natural transformation $\alpha\colon F\Rightarrow G$ has a right adjoint in $\Funlax(\cU,\cV)$ if and only if it is a strong natural transformation and $\alpha_X\colon FX\to GX$ has a right adjoint for every $X\in \cU$.
\end{proposition}

This has the following corollary for adjunctions of strong natural transformation:

\begin{corollary}[{\cite[Corollary 5.2.12]{abellan2024adjunctions}}]\label{c:adj-nat}
  A strong natural transformation $\alpha\colon F\Rightarrow G$ has a right adjoint in $\Fun(\cU,\cV)$ if and only if $\alpha_X\colon FX\to GX$ has a right adjoint for every $X\in \cU$ and for every $f\colon X\to Y$ the corresponding Beck-Chevalley 2-morphism is invertible.
\end{corollary}

For our purposes, it would be helpful to give a different formulation of \cref{p:adj-lax-nat}.
Denote by $\Adj\in \Cat_2$ the walking adjunction, which is the free pair of adjoint morphisms, with $l\colon \Delta^1\to \Adj$ picking the left adjoint.
Given a 2-category $\cV\in \Cat_2$, precomposition with $l$ defines a functor 
$$l^*\colon \Fun^\oplax(\Adj,\cV)\to \Fun^\oplax(\Delta^1,\cV).$$

\begin{lemma}[{\cite[Theorem 1.1]{Haugseng2021LaxTransformations}}]\label{l:adj-is-ff}
  For $\cV\in \Cat_2$, precomposition with $l$ lifts to a fully faithful functor:
  \[\begin{tikzcd}
    & {\Fun(\Delta^1,\cV)} \\
    {\Fun^\oplax(\Adj,\cV)} & {\Fun^\oplax(\Delta^1,\cV)}.
    \arrow[from=1-2, to=2-2]
    \arrow[dashed, hook, from=2-1, to=1-2]
    \arrow["{l^*}", from=2-1, to=2-2]
  \end{tikzcd}\]
\end{lemma}

\begin{proof}
  While \cite{Haugseng2021LaxTransformations} gives an independent proof, we will also explain how this follows from \cref{p:adj-lax-nat}.
  Note that, as $\Fun(\Delta^1,\cV)\to \Fun^\oplax(\Delta^1,\cV)$ is a monomorphism, if such a lift exists then it is unique.
  We prove that the lift exists using Yoneda's Lemma.
  Let $\cU\in \Cat_2$, and apply $\Map(\cU,-)$ to the above diagram. 
  After swapping the order with adjunctions, we get the following lifting problem:
  \[\begin{tikzcd}
    & {\Map(\Delta^1,\Fun(\cU,\cV))} \\
    {\Map(\Adj,\Funlax(\cU,\cV))} & {\Map(\Delta^1,\Funlax(\cU,\cV))}
    \arrow[from=1-2, to=2-2]
    \arrow[dashed, from=2-1, to=1-2]
    \arrow["{{l^*}}", from=2-1, to=2-2]
  \end{tikzcd}\]
   where a lift exists from \cref{p:adj-lax-nat} because left adjoint lax natural transformations are strong. 
   The fact that the lift is fully faithful then follows from the fact that, among the strong natural transformations, being left adjoint in $\Funlax(\cU,\cV)$ is an object-wise property.
\end{proof}

\subsection{Symmetric monoidal lax natural transformations}

In this subsection we prove a symmetric monoidal version of \cref{p:adj-lax-nat}.
Notice that for $\cW\in \Cat_2$, the functor $\Fun(\cW,-)$ commutes with products, and hence if $\cV\in \Catsm_2$ is symmetric monoidal then $\Fun(\cW,\cV)$ inherits a (pointwise) symmetric monoidal structure.
Assume also that $\cU\in \Catsm_2$.
Using the above structure, the 2-category of symmetric monoidal functors $\Funo(\cU,\cV)$ can be characterized by the universal property
$$\Map(\cW,\Funo(\cU,\cV))\simeq \Mapo(\cU,\Fun(\cW,\cV)).$$
Similarly, the functor $\Funbilax(\cW,-)$ commutes with products, and hence $\Funbilax(\cW,\cV)$ inherits a symmetric monoidal structure.
Following this observation, \cite[Definition 2.17]{Carmeli2022CharactersAT} gave the following definition:

\begin{definition}
  Let $\cU,\cV\in \Catsm_2$, the 2-categories $\Funolax(\cU,\cV)$ and $\Funooplax(\cU,\cV)$ are defined by the universal properties that for every $\cW\in \Cat_2$ 
  $$\Map(\cW,\Funolax(\cU,\cV))\simeq \Mapo(\cU,\Funoplax(\cW,\cV)),$$
  $$\Map(\cW,\Funooplax(\cU,\cV))\simeq \Mapo(\cU,\Funlax(\cW,\cV)).$$
The objects of $\Funobilax(\cU,\cV)$ are strong symmetric monoidal functors and the morphisms are (op)lax natural transformation which are strong symmetric monoidal.
\end{definition}

\begin{warning-thm}
  Despite the confusing notation, the objects of $\Funolax(\cU,\cV)$ are \emph{not} lax symmetric monoidal functors.
\end{warning-thm}

The above universal properties can be made internal, as isomorphisms of functor 2-categories instead of mapping spaces, by using Yoneda's lemma.

\begin{lemma}\label{l:Funo-internal-universal-propertey}
  Let $\cU,\cV\in \Catsm_2$ and $\cW\in \Cat_2$. 
  There are isomorphisms
  $$\Fun(\cW,\Funo(\cU,\cV))\simeq \Funo(\cU,\Fun(\cW,\cV)),$$
  $$\Funoplax(\cW,\Funolax(\cU,\cV))\simeq \Funolax(\cU,\Funoplax(\cW,\cV)),$$
  $$\Funlax(\cW,\Funooplax(\cU,\cV))\simeq \Funooplax(\cU,\Funlax(\cW,\cV)).$$
\end{lemma}

\begin{proof}
  We will prove the last isomorphism, the rest are similar.
  For every $\cT\in \Cat_2$, there are the following isomorphisms:
  \begin{align*}
    \Map(\cT,\Funlax(\cW,\Funooplax(\cU,\cV)))
    &\simeq  \Map(\cT\xlax \cW,\Funooplax(\cU,\cV))\\
    &\simeq  \Mapo(\cU,\Funlax(\cT\xlax \cW,\cV))\\
    &\simeq \Mapo(\cU,\Funlax(\cT,\Funlax(\cW,\cV)))\\
    &\simeq  \Map(\cT,\Funooplax(\cU,\Funlax(\cW,\cV))).\\
  \end{align*}
  The third isomorphism $\Funlax(\cT\xlax \cW,\cV)\simeq \Funlax(\cT,\Funlax(\cW,\cV))$ follows from the associativity of the Gray product.
\end{proof}

By construction, the inclusion $\Fun(\cW,\cV)\to \Funbilax(\cW,\cV)$ is symmetric monoidal, so it induces a functor $\Funo(\cU,\cV)\to \Funobilax(\cU,\cV)$.

\begin{lemma}\label{l:funo-is-locally-full}
  For every $\cU,\cV\in \Catsm_2$, the functor $\Funo(\cU,\cV)\to \Funobilax(\cU,\cV)$ is a locally full inclusion.
\end{lemma}

\begin{proof}
  We will consider only $\Funo(\cU,\cV)\to \Funolax(\cU,\cV)$, and show that it is right orthogonal to the map $A\to B$ which stands either for $\pt\sqcup \pt\to \pt$ or for $\partial C_2\to C_2$:
   \[\begin{tikzcd}
      A & {\Funo(\cU,\cV)} \\
      B & {\Funolax(\cU,\cV)}.
      \arrow[from=1-1, to=1-2]
      \arrow[from=1-1, to=2-1]
      \arrow[from=1-2, to=2-2]
      \arrow[dashed, from=2-1, to=1-2]
      \arrow[from=2-1, to=2-2]
    \end{tikzcd}\]
    Passing through the universal properties, this is equivalent to the following square being a pullback in $\Catsm_2$:
    \[\begin{tikzcd}
      {\Fun(B,\cV)} & {\Funoplax(B,\cV)} \\
      {\Fun(A,\cV)} & {\Funoplax(A,\cV)}.
      \arrow[from=1-1, to=1-2]
      \arrow[from=1-1, to=2-1]
      \arrow[from=1-2, to=2-2]
      \arrow[from=2-1, to=2-2]
    \end{tikzcd}\]
    As the forgetful functor is conservative and preserves limits, it is enough to check that the above square is a pullback in $\Cat_2$.
    However, the fact that $\Fun(\cU,\cV)\to \Funlax(\cU,\cV)$ is a locally full inclusion gives us the lift
    \[\begin{tikzcd}
      A & {\Fun(\cU,\cV)} \\
      B & {\Funlax(\cU,\cV)}
      \arrow[from=1-1, to=1-2]
      \arrow[from=1-1, to=2-1]
      \arrow[from=1-2, to=2-2]
      \arrow[dashed, from=2-1, to=1-2]
      \arrow[from=2-1, to=2-2]
    \end{tikzcd}\]
    which implies the same pullback in $\Cat_2$.
\end{proof}

\begin{observation}\label{o:sm-adj-is-ff}
  Let $\cV\in \Catsm_2$, and consider the diagram from \cref{l:adj-is-ff}:
  \[\begin{tikzcd}
    & {\Fun(\Delta^1,\cV)} \\
    {\Fun^\oplax(\Adj,\cV)} & {\Fun^\oplax(\Delta^1,\cV)}.
    \arrow[from=1-2, to=2-2]
    \arrow[dashed, hook, from=2-1, to=1-2]
    \arrow["{l^*}", from=2-1, to=2-2]
  \end{tikzcd}\]
  All of the above maps are monomorphisms, and the solid arrows are symmetric monoidal, so the dashed arrow is also (uniquely) symmetric monoidal.
\end{observation}
Using this observation, we can prove the symmetric monoidal version of \cref{p:adj-lax-nat}:

\begin{proposition}\label{p:sm-adj-lax-nat}
  Let $\cU,\cV\in \Catsm_2$ and let $F,G\colon \cU\to \cV$ be symmetric monoidal functors.
  A symmetric monoidal lax natural transformation $\alpha\colon F\Rightarrow G$ has a right adjoint in $\Funolax(\cU,\cV)$ if and only if it has a right adjoint in $\Funlax(\cU,\cV)$, namely it is a strong natural transformation and $\alpha_X\colon F(X)\to G(X)$ has a right adjoint for every $X\in \cU$. 
\end{proposition}

\begin{proof}
  The ``only if'' direction follows from \cref{p:adj-lax-nat}, by forgetting the symmetric monoidal structure. 
  For the ``if'' direction, suppose $\alpha$ is strong and $\alpha_X$ is left adjoint for every $X\in \cV$. 
  We want to show that $\alpha$ is in the image of the map
  $$\Map(\Adj,\Fun^\lax_\otimes(\cU,\cV))\oto{l^*}\Map(\Delta^1,\Fun^\lax_\otimes(\cU,\cV))$$
  which corresponds to the map
  $$\Map_\otimes(\cU,\Fun^\oplax(\Adj,\cV))\oto{l^*}\Map_\otimes(\cU,\Fun^\oplax(\Delta^1,\cV)).$$
  By the assumption that $\alpha$ is a strong natural transformation, $\alpha$ is in the image of
  $$\Map_\otimes(\cU,\Fun(\Delta^1,\cV))\to \Map_\otimes(\cU,\Fun^\oplax(\Delta^1,\cV)).$$
  However, by \cref{l:adj-is-ff,o:sm-adj-is-ff} there is a fully faithful symmetric monoidal embedding
  $$\Fun^\oplax(\Adj,\cV)\hookrightarrow\Fun(\Delta^1,\cV)$$
  thus to see that $\alpha$ is in the image of $\Map_\otimes(\cU,\Fun^\oplax(\Adj,\cV))$ it is enough to notice that $\alpha$ is object-wise left adjoint.
\end{proof}

Using the fact that $\Funo(\cU,\cV)\to \Funolax(\cU,\cV)$ is a locally full inclusion (\cref{l:funo-is-locally-full}), we get the following corollary for symmetric monoidal strong natural transformations:

\begin{corollary}\label{c:sm-adj-nat}
  A symmetric monoidal strong natural transformation $\alpha\colon F\Rightarrow G$ has a right adjoint in $\Funo(\cU,\cV)$ if and only if it has a right adjoint in $\Fun(\cU,\cV)$, namely $\alpha_X\colon FX\to GX$ has a right adjoint for every $X\in \cU$ and for every $f\colon X\to Y$ the corresponding Beck-Chevalley 2-morphism is invertible.
\end{corollary}

\subsection{Bimodules in 2-categories}

In this subsection we deduce the projection formula for bimodule maps.
Before considering bimodules, we will first study algebras over an arbitrary operad, using the symmetric monoidal envelope.
Let $\cO$ be an operad and $\cC$ a symmetric monoidal 1-category, the symmetric monoidal envelope $\Env(\cO)$ was constructed in \cite[Construction 2.2.4.1]{ha} such that
$$\Alg_\cO(\cC)\simeq \Funo(\Env(\cO),\cC).$$

\begin{example}\label{e:bimodule-envelope}
  Our main example of interest is the bimodule operad $\BM$.
  The symmetric monoidal envelope of $\BM$ is a (1,1)-category, which we describe explicitly:
  \begin{itemize}
    \item The objects of $\Env(\BM)$ are finite sets $X$ with a partition into three parts
    $$X=X_L\sqcup X_M\sqcup X_R.$$
    \item A morphism $X\to Y$ in $\Env(\BM)$ is a function $f\colon X\to Y$ together with linear orderings of the fibers $f^{-1}(y)$ for every $y\in Y$, such that:
    \begin{enumerate}
      \item for $l\in Y_L$, $f^{-1}(l)\subseteq X_L$,
      \item for $r\in Y_R$, $f^{-1}(r)\subseteq X_R$, and
      \item for $m\in Y_M$, $f^{-1}(m)$ is partitioned into ordered segments 
      $$f^{-1}(m)\cap X_L
      \ <\
      f^{-1}(m)\cap X_M
      \ <\ 
      f^{-1}(m)\cap X_R$$
      and $f^{-1}(m)\cap X_M$ has exactly 1 element.
    \end{enumerate}
  \end{itemize}
  The symmetric monoidal structure of $\Env(\BM)$ is given by disjoint union of sets, respecting the partitions.
  Thus, a symmetric monoidal functor $\Env(\BM)\to \cC$ is of the form
  $$X\mapsto A^{\otimes X_L}\otimes M^{\otimes X_M}\otimes B^{\otimes X_R}$$
  for some $A,M,B\in \cC$, and the morphisms encode the structure of an associative algebra on $A$ and $B$ and an $A$-$B$-bimodule on $M$.
\end{example}

\begin{definition}
  Let $\cU\in \Catsm_2$ be a symmetric monoidal 2-category and $\cO$ an operad.
  The 2-category of $\cO$-algebras in $\cU$ is defined as
  $$\Alg_\cO(\cU):=\Funo(\Env(\cO),\cU).$$
  Similarly, $\cO$-algebras with (op)lax $\cO$-algebra maps are defined as 
  $$\Alg^\bilax_\cO(\cU):=\Funobilax(\Env(\cO),\cU).$$
\end{definition}

\begin{remark}
  In the special case of the 2-category of 1-categories, $\Alg^\bilax_\cO(\Cat_1)$ can also be described in terms of fibrations over $\cO$, see \cite{haugseng2023lax}.
  We will not require this perspective, so we will not prove the equivalence of the two definitions.
\end{remark}

As an immediate corollary of \cref{l:Funo-internal-universal-propertey}, we can consider diagrams of $\cO$-algebras as $\cO$-algebras in the category of diagrams.

\begin{corollary}\label{c:algebras-in-lax-arrow-category}
  Let $\cO$ be an operad, $\cU\in \Catsm_2$ and $\cV\in \Cat_2$.
  There are isomorphisms of 2-categories
  $$\Fun(\cV,\Alg_\cO(\cU))\simeq \Alg_\cO(\Fun(\cV,\cU)).$$
  $$\Funlax(\cV,\Alg_\cO^\oplax(\cU))\simeq \Alg_\cO^\oplax(\Funlax(\cV,\cU)).$$
\end{corollary}

The cases we care about are associative algebras, commutative algebras, and bimodules, which are algebras over the following operads \cite{ha}:
\begin{align*}
  \Alg(\cU):=&\Alg_{\EE_1}(\cU), && 
  \quad \Alg^\bilax(\cU):=\Alg^\bilax_{\EE_1}(\cU),\\
  \CAlg(\cU):=&\Alg_{\EE_\infty}(\cU), && 
  \ \CAlg^\bilax(\cU):=\Alg^\bilax_{\EE_\infty}(\cU),\\
  \BMod(\cU):=&\Alg_\BM(\cU), && 
  \BMod^\bilax(\cU):=\Alg^\bilax_\BM(\cU).
\end{align*}
The bimodule operad comes equipped with inclusions of $\EE_1$-operads from the left and right:
$$\EE_1\to \BM\leftarrow \EE_1,$$
so that for $A,B\in \Alg(\cU)$, the category of $A$-$B$-bimodules is given by the pullback
\begin{align*}
  \Bmod[A]{B}(\cU)&:=\{A\}\times_{\Alg(\cU)}\BMod(\cU)\times_{\Alg(\cU)} \{B\},\\
  \Bmod[A]{B}^\bilax(\cU)&:=\{A\}\times_{\Alg^\bilax(\cU)}\BMod^\bilax(\cU)\times_{\Alg^\bilax(\cU)} \{B\}.
\end{align*}
Left and right module categories are obtained by setting one side to be the unit:
\begin{align*}
  \LMod_A(\cU)&:=\Bmod[A]{\one}(\cU),&&
  \LMod^\bilax_A(\cU):=\Bmod[A]{\one}^\bilax(\cU),\\
  \RMod_A(\cU)&:=\Bmod[\one]{A}(\cU),&&
  \RMod_A^\bilax(\cU):=\Bmod[\one]{A}^\bilax(\cU).
\end{align*}
There is also a map of operads $\BM\to \EE_1$ such that the induced functor $\Alg(\cU)\to \BMod(\cU)$ equips each algebra $A$ with its canonical $A$-$A$-bimodule structure.

For $\cV\in \Cat_2$, consider the projection to the point $p\colon \cV\to \pt$. 
The induced symmetric monoidal functor $p^*\colon \cU\to \Fun(\cV,\cU)$ sends $X\in \cU$ to the constant diagram on $X$. 

\begin{corollary}\label{c:modules-in-lax-arrow-category}
  For $A,B\in \Alg(\cU)$, there are isomorphism 
  $$\Fun(\cV,\Bmod[A]{B}(\cU))\simeq \Bmod[p^*A]{p^*B}(\Fun(\cV,\cU)).$$
  $$\Funlax(\cV,\Bmod[A]{B}^\oplax(\cU))\simeq \Bmod[p^*A]{p^*B}^\oplax(\Funlax(\cV,\cU)).$$
\end{corollary}

\begin{proof}
  Both $\Fun(\cV,-)$ and $\Funlax(\cV,-)$ commute with pullbacks.
  Thus, the result follows from \cref{c:algebras-in-lax-arrow-category} for $\cO=\BM$.
\end{proof}

Fix algebras $A,B\in \Alg(\cU)$.
For a bimodule $M\in \Bmod[A]{B}(\cU)$, denote the bi-action on $M$ by 
$$\beta_M\colon A\otimes M\otimes B\to M.$$
Suppose $f\colon M\to N\in  \Bmod[A]{B}(\cU)$ is a map of bimodules.
As part of the bimodule structure, $f$ in particular commutes with the bi-action:
\[\begin{tikzcd}
	{A\otimes M\otimes B} && {A\otimes N\otimes B} \\
	M && N.
	\arrow["{\id_A\otimes f\otimes \id_B}", from=1-1, to=1-3]
	\arrow["{\beta_M}"', from=1-1, to=2-1]
	\arrow["{\beta_N}", from=1-3, to=2-3]
	\arrow["f"', from=2-1, to=2-3]
\end{tikzcd}\]
Similarly, if $f\colon M\to N\in  \Bmod[A]{B}^\lax(\cU)$ is a lax linear map of bimodules, then it commutes with the bi-action up to a specified 2-morphism:
\[\begin{tikzcd}
	{A\otimes M\otimes B} && {A\otimes N\otimes B} \\
	M && N.
	\arrow["{\id_A\otimes f\otimes \id_B}", from=1-1, to=1-3]
	\arrow["{\beta_M}"', from=1-1, to=2-1]
	\arrow[shorten <=24pt, shorten >=24pt, Rightarrow, from=1-3, to=2-1]
	\arrow["{\beta_N}", from=1-3, to=2-3]
	\arrow["f"', from=2-1, to=2-3]
\end{tikzcd}\]

\begin{proposition}[projection formula]\label{p:projection-formula}
  Let $\cU\in \Catsm_2$.
  Given $A,B\in \Alg(\cU)$ and a map of bimodules $f\colon M\to N\in \Bmod[A]{B}(\cU)$, if $f$ has a right adjoint in $\cU$ then it has a right adjoint in $\Bmod[A]{B}^\lax(\cU)$, where $f^R$ commutes with the bi-action up to the 2-morphism
  $$\beta_M(\id_A\otimes f^R\otimes \id_B)\xRightarrow{u_f}f^Rf\beta_M (\id_A\otimes f^R\otimes \id_B)\xRightarrow{\sim} f^R\beta_N (\id_A\otimes ff^R\otimes \id_B)\xRightarrow{c_f}f^R\beta_N.$$
  Moreover, if the above composition is an isomorphism then the adjunction is in $\Bmod[A]{B}(\cU)$.
\end{proposition}

\begin{proof}
  It suffices to show that $f$ has a right adjoint in $\BMod^\lax(\cU)=\Funolax(\Env(\BM),\cU)$.
  Recall the explicit description of $\Env(\BM)$ given in \cref{e:bimodule-envelope}.
  By \cref{p:sm-adj-lax-nat}, we need to show that for every $X\in \Env(\BM)$, the map
  $$f_X\colon A^{\otimes X_L}\otimes M^{\otimes X_M}\otimes B^{\otimes X_R}\to A^{\otimes X_L}\otimes N^{\otimes X_M}\otimes B^{\otimes X_R}$$
  has a right adjoint.
  Indeed, $f_X$ is a tensor product of copies of $\id_A$, $f$, and $\id_B$, which all have right adjoints.
  The lax linear structure on $f^R$ then comes from the Beck-Chevalley 2-morphism, which is the composition described above.

  Assume now that the composition above is an isomorphism, by \cref{c:sm-adj-nat} we need to show that all Beck-Chevalley 2-morphisms of $f^R$ are isomorphisms.
  However, every map in $\Env(\BM)$ is produced by tensor products and compositions from the following maps:
  $$\{l,m,r\}\to \{m\},$$
  $$\emptyset \to \{l\},
  \hspace{15pt}
  \{l_0,l_1\}\to \{l\},$$
  $$\emptyset \to \{r\},
  \hspace{15pt}
  \{r_0,r_1\}\to \{r\},$$
  with the implied partitions and orderings of the fibers.
  Thus, it is enough to show that the Beck-Chevalley 2-morphisms are invertible for the squares induced from these five maps.
  For the first map this is exactly the assumption, and for the rest of the maps we get degenerate squares.
\end{proof}

\section{Rigid algebras}\label{s:rigid-algebras}

We now consider rigid algebras in a general symmetric monoidal 2-category $\cU\in \Catsm_2$.

\begin{definition}
  An algebra $A\in \Alg(\cU)$ is called \emph{rigid} if:
  \begin{enumerate}
    \item The unit $\eta\colon \one\to A$ has a right adjoint $\epsilon\colon A\to \one$.
    \item The multiplication $\mu\colon A\otimes A\to A$ has a right adjoint $\delta\colon A\to A\otimes A$ in $\Bmod[A]{A}(\cU)$.
  \end{enumerate}
  Denote by $\Algrig(\cU)\subseteq\Alg(\cU)$ the full 2-subcategory of rigid algebras.
\end{definition}

Note that the above definition specializes to the definition of rigid categories in $\Prl_B$.
In this section, we study rigid algebras in full generality.

Our first goal is to demonstrate that rigid algebras are a categorification of Frobenius algebras.
In a symmetric monoidal 1-category $\cC$, a Frobenius algebra is an algebra $A\in \Alg(\cC)$ equipped with additional structure that, in particular, makes $A$ self-dual.
When passing from 1-categories to  2-categories, structure maps can often be obtained as adjoints, turning what was extra data into a property. 
Rigid algebras arise from Frobenius algebras through precisely this process.
To set the stage, we begin by reviewing the 1-categorical notions of duality and Frobenius algebras.

Our second goal is to study algebra maps between rigid algebras $f\colon A\to B\in \Algrig(\cU)$.
Using the terminology introduced in \cref{s:projection-formula}, we show that $f$ is a rigid algebra in the lax arrow category $\Funlax(\Delta^1,\cU)$.
In particular, $f$ is self-dual in the lax arrow category, which implies that $f$ is left adjoint to its transpose.
Moreover, we deduce that all 2-morphisms in $\Algrig(\cU)$ are invertible, that is, $\Algrig(\cU)$ is in fact a 1-category.

\begin{remark}
  Many of the definitions and results in this section can be adapted to a (not necessarily symmetric) monoidal setting.
  The main difference would be that instead of duals and transposes we would need to keep track of left and right duals and left and right transposes.
  To keep the statements simpler, we will always assume a symmetric monoidal structure.
\end{remark}

\subsection{Duality}

Given a vector space $V$ over $k$, the dual vector space $\dual{V}:=\hom_k(V,k)$ comes with an evaluation map
$$\ev\colon \dual{V}\otimes_k V\to k.$$
If moreover $V$ is finite dimensional, with basis $e_1,\dots,e_n$, then we can construct a coevaluation map
$$\coev\colon k\to V\otimes_k \dual{V}$$
which chooses the identity matrix $\sum_{i=1}^ne_i\otimes \dual{e_i}\in V\otimes_k \dual{V}$, and the evaluation and coevaluation maps satisfy certain zigzag identities.
In fact, the finite dimensional vector spaces are precisely characterized by the existence of such structure.

Fix a symmetric monoidal 1-category $\cC$.
This notion of duality has been generalized to objects in $\cC$, obtaining a certain analog of ``finite'' objects.

\begin{definition}
  An object $X\in \cC$ is called \emph{dualizable} if there exist:
  \begin{enumerate}
      \item  An object $\dual{X} \in \cC$ called the \emph{dual},
      \item maps called \emph{evaluation} and \emph{coevaluation} respectively
      $$\ev \colon \dual{X}\otimes X\to \one
      \hspace{60pt}
      \coev \colon \one \to X\otimes \dual{X},$$
      \item 2-isomorphisms called the \emph{zigzag identities}
      \[\begin{tikzcd}
          X \arrow[r, "\coev\otimes \id"] \arrow[rd, equal] & X\otimes X^\vee\otimes X \arrow[d, "\id\otimes\ev"] &  & X^\vee \arrow[r, "\id\otimes\coev"] \arrow[rd, equal] & X^\vee\otimes X\otimes X^\vee \arrow[d, "\ev\otimes\id"] \\
          & X    &  &    & X^\vee.   
          \end{tikzcd}\]
  \end{enumerate}
\end{definition}

If $X\in \cC$ is dualizable, then the space of duality data as above is contractible, hence we may speak of \emph{the} dual, evaluation and coevaluation.
Additionally, if we specify an object $\dual{X}$ and an evaluation map $\ev\colon \dual{X}\otimes X\to \one$ (respectively, a coevaluation map $\coev\colon \one\to X\otimes\dual{X}$), then the space of extensions to a complete duality data is contractible (\cite[Lemma 4.6.1.10]{ha}); this is not necessarily true if we specify only $\dual{X}$ or both $\ev$ and $\coev$.
By symmetry, if $\dual{X}$ is the dual of $X$ then $X$ is the dual of $\dual{X}$. 

\begin{definition}
  Given a map $f\colon X\to Y$ between dualizable objects $X,Y\in \cC$, the \emph{transpose} map $\trans{f}\colon \dual{Y}\to\dual{X}$ is defined as the composition
  $$\dual{Y}\oto{\coev_X}\dual{Y}\otimes X\otimes \dual{X}\oto{f}\dual{Y}\otimes Y\otimes \dual{X}\oto{\ev_Y}\dual{X}.$$
\end{definition}

The full subcategory $\cC^\dbl\subseteq\cC$ on dualizable objects is closed under tensor products, and there is a symmetric monoidal involution $\cC^\dbl\isoto (\cC^\dbl)^\op$ given by taking duals of objects and transposes of maps (\cite[Proposition 3.2.4]{lurie2018ellipticI}).

The notion of duality can also be generalized to the context of bimodules.
Assume for now that $\cC$ has geometric realizations that commute with the tensor product in each variable, and let $A,B,C\in \Alg(\cC)$ be algebras.
Given bimodules  $M\in \Bmod[A]{B}(\cC)$ and $N\in \Bmod[B]{C}(\cC)$, their \emph{relative tensor product} over $B$
$$M\otimes_B N\in \Bmod[A]{C}(\cC)$$
is the colimit of the bar complex:
$$M\otimes_B N:=\colim(\cdots\begin{tikzcd}
	{M\otimes B\otimes B\otimes N} & {M\otimes B\otimes N} & {M\otimes N}
	\arrow[from=1-1, to=1-2]
	\arrow[shift left=2, from=1-1, to=1-2]
	\arrow[shift right=2, from=1-1, to=1-2]
	\arrow[shift right, from=1-2, to=1-3]
	\arrow[shift left, from=1-2, to=1-3]
\end{tikzcd}).$$
We expand on this construction in \cref{d:bar-construction}.

\begin{definition}
    Suppose $\cC$ has geometric realizations that commute with the tensor product in each variable.
    A bimodule $M\in \Bmod[A]{B}(\cC)$ is \emph{right dualizable} if there exist:
    \begin{enumerate}
        \item  a bimodule $N\in \Bmod[B]{A}(\cC)$ called the \emph{right dual},
        \item evaluation and coevaluation maps
        $$\ev \colon N\otimes_A M\to B \in \Bmod[B]{B}(\cC)
        \hspace{40pt}
        \coev \colon A \to M\otimes_B N\in \Bmod[A]{A}(\cC),$$
        \item zigzag identities
        \[\begin{tikzcd}[column sep=large]
          M & {M\otimes_B N\otimes_A M} && N & {N\otimes_A M\otimes_B N} \\
          & M &&& {N.}
          \arrow["{\coev\otimes_A \id}", from=1-1, to=1-2]
          \arrow[equals, from=1-1, to=2-2]
          \arrow["{\id\otimes_B\ev}", from=1-2, to=2-2]
          \arrow["{\id\otimes_A\coev}", from=1-4, to=1-5]
          \arrow[equals, from=1-4, to=2-5]
          \arrow["{\ev\otimes_B\id}", from=1-5, to=2-5]
        \end{tikzcd}\]
    \end{enumerate}
\end{definition}

We are specifically interested in the case where $M\in \LMod_A(\cC)$ has a right dual $N\in \RMod_A(\cC)$, in which case the evaluation and coevaluation are of the form
$$\ev \colon N\otimes_A M\to \one \in \cC
\hspace{40pt}
\coev \colon A \to M\otimes N\in \Bmod[A]{A}(\cC).$$
The following lemma relates duality of modules to regular duality:

\begin{lemma}[{\cite[Proposition 4.6.2.13]{ha}}]\label{l:dual-iff-bmod-dual}
  Suppose $\cC$ has geometric realizations that commute with the tensor product in each variable.
  For a left $A$-module $M\in\LMod_A(\cC)$, the following are equivalent:
  \begin{enumerate}
    \item $M$ is right dualizable in $\LMod_A(\cC)$.
    \item For every algebra map $B\to A$, the restriction of scalars $M\in \LMod_B(\cC)$ is right dualizable.
    \item $M$ is dualizable in $\cC$.
  \end{enumerate}
\end{lemma}

Explicitly, if $M\in \LMod_A(\cC)$ has a right dual $N\in \RMod_A(\cC)$, then by \cite[Remark 4.6.2.12]{ha} the dual of $M$ in $\cC$ is the underlying object of $N$ and the evaluation and coevaluation maps are given by the compositions
$$N\otimes M\to N\otimes_A M\to \one,
\hspace{40pt}
\one\to A\to  M\otimes N.$$
While we needed geometric realizations to set up the theory of duality in bimodules, the conclusion that $\dual{M}$ has a right $A$-module structure can be phrased without them.
By passing to the presheaf category, we may drop the assumption on geometric realizations.

\begin{observation}\label{o:add-geometric-realizations}
  Consider the symmetric monoidal Yoneda embedding $\cC\hookrightarrow \PSh(\cC)$, where the symmetric monoidal structure on the category of presheaves is given by the Day convolution \cite{glasman2016day,ha}.
  For $A,B\in \Alg(\cC)$, we get a fully faithful embedding 
  $$\Bmod[A]{B}(\cC)\hookrightarrow\Bmod[A]{B}(\PSh(\cC)).$$
  In contexts involving $A$-$B$-bimodules in $\cC$, we may use this embedding to assume that $\cC$ has geometric realizations that commute with the tensor product in each variable.
\end{observation}

Using the above observation, we get the following corollary of \cref{l:dual-iff-bmod-dual}:

\begin{corollary}\label{c:dual-of-lmod-is-rmod}
  Let $A\in \Alg(\cC)$ and suppose $M\in \LMod_A(\cC)$ is such that $M$ is dualizable in $\cC$.
  Then $\dual{M}$ has a natural structure of a right $A$-module.
\end{corollary}

Similarly, the transpose of a left $A$-module map between dualizable modules has the structure of a right $A$-module map:

\begin{lemma}\label{l:trans-is-rmod}
  Let $A\in \Alg(\cC)$ and consider $f\colon M\to N\in \LMod_A(\cC)$ such that $M$ and $N$ are dualizable in $\cC$, then $\trans{f}\colon \dual{N}\to \dual{M}$ has a natural structure of a right $A$-module map.
\end{lemma}

\begin{proof}
  Use \cref{o:add-geometric-realizations} to assume that $\cC$ has geometric realizations that commute with the tensor product in each variable.
  The result then follows from the fact that the transpose of $f$ induced from the duality in $\cC$ agrees with the transpose induced from the duality of bimodules.
  Explicitly, the following diagram commutes:
  \[\begin{tikzcd}[column sep=huge]
    {\dual{N}} & {\dual{N}\otimes M\otimes\dual{M}} & {\dual{N}\otimes N\otimes\dual{M}} & {\dual{M}} \\
    {\dual{N}\otimes_A A} & {\dual{N}\otimes_A M\otimes\dual{M}} & {\dual{N}\otimes_A N\otimes\dual{M}} & {\dual{M},}
    \arrow[from=1-1, to=1-2]
    \arrow["\sim"', from=1-1, to=2-1]
    \arrow["f", from=1-2, to=1-3]
    \arrow[from=1-2, to=2-2]
    \arrow[from=1-3, to=1-4]
    \arrow[from=1-3, to=2-3]
    \arrow[equals, from=1-4, to=2-4]
    \arrow[from=2-1, to=2-2]
    \arrow["f"', from=2-2, to=2-3]
    \arrow[from=2-3, to=2-4]
  \end{tikzcd}\]
  where the top row is the transpose of $f$ and the bottom row is a composition of right $A$-module maps.
\end{proof}

\subsection{Frobenius algebras}

Frobenius algebras are self-dual algebras in which the evaluation map arises from the multiplication.
As an example, consider $M_n(k)$ the algebra of $n\times n$ matrices over a field $k$.
Composing the multiplication of matrices with the trace map
$$M_n(k)\otimes_k M_n(k)\to M_n(k)\oto{\tr} k$$
produces an evaluation map exhibiting $M_n(k)$ as self-dual.

\begin{definition}
  A \emph{Frobenius algebra} in $\cC$ is an algebra $A\in \Alg(\cC)$ together with a map $\epsilon \colon A\to \one$, such that the composition
  $$A\otimes A\xrightarrow{\mu}A\xrightarrow{\epsilon}\one$$
  is an evaluation exhibiting $A$ as self-dual, where $\mu$ is the multiplication of $A$.
\end{definition}

\begin{example}
  Beside the example of matrices given above, another class of examples for Frobenius algebras comes from the 1-category of (co)spans, see \cref{r:cospan-can-is-frob}.
\end{example}

An algebra $A\in \Alg(\cC)$ can be viewed as both a left and a right module over itself.
A Frobenius algebra is then precisely a self-dual algebra such that the dual of $A$ as a left $A$-module is $A$ as a right $A$-module:

\begin{lemma}[{\cite[Proposition 4.6.5.2]{ha}}]\label{l:frob-is-mod-dual}
    Suppose $\cC$ has geometric realizations that commute with the tensor product in each variable.
    Given an algebra $A\in \Alg(\cC)$ and a map $\epsilon\colon A\to \one$, the following are equivalent:
    \begin{enumerate}
        \item $(A,\epsilon)$ is a Frobenius algebra.
        \item Considering $\epsilon$ as a map $A\otimes_A A\to \one$, it is an evaluation exhibiting $A\in \RMod_A(\cC)$ as right dual to $A\in \LMod_A(\cC)$.
    \end{enumerate}
\end{lemma}

Elaborating on condition (2) above, it is equivalent to having a coevaluation map, which we denote
$$\delta\colon A\to A\otimes A \in \Bmod[A]{A}(\cC),$$
together with zigzag identities
$$(\epsilon\otimes \id_A)\delta\simeq \id_A\simeq (\id_A\otimes \epsilon)\delta.$$
Notice that the only relative tensor products we needed to phrase $\epsilon$ and $\delta$ were trivial, and they exist universally.
Thus, we may use \cref{o:add-geometric-realizations} to drop the assumption on geometric realizations.

\begin{corollary}\label{c:frob-equiv-explicit}
  For an algebra $A\in \Alg(\cC)$ with a map $\epsilon\colon A\to \one$, the following are equivalent:
  \begin{enumerate}
    \item $(A,\epsilon)$ is a Frobenius algebra.
    \item There exists a map $\delta\colon A\to A\otimes A\in \Bmod[A]{A}(\cC)$ and 2-isomorphisms
    $$(\epsilon\otimes \id_A)\delta\simeq \id_A\simeq (\id_A\otimes \epsilon)\delta.$$
  \end{enumerate}
\end{corollary}

\begin{remark}
  Let $A\in \Alg(\cC)$ be an algebra with unit $\eta\colon \one\to A$ and multiplication $\mu\colon A\otimes A\to A$. 
  Suppose $\epsilon\colon A\to \one$ and $\delta\colon A\to A\otimes A$ exhibit $A$ as a Frobenius algebra as in \cref{c:frob-equiv-explicit}, then the self-duality of $A$ is such that $\epsilon\simeq\trans{\eta}$ and $\delta\simeq\trans{\mu}$. 
  Moreover, the symmetric monoidal involution $\cC^\dbl\isoto (\cC^\dbl)^\op$ endows $A$ with a coalgebra structure with counit $\epsilon$ and comultiplication $\delta$; the zigzag identities in \cref{c:frob-equiv-explicit} are then the counitality relations.
  We will henceforth refer to $\epsilon$ as the \emph{counit} and $\delta$ the \emph{comultiplication} of $A$.
\end{remark}

Suppose $A,B\in \Alg(\cC)$ are algebras and $f\colon A\to B$ is a map of algebras.
In particular, $f$ induces both a left and a right $A$-module structure on $B$ such that $f$ is a map of left and right $A$-modules. 

\begin{corollary}\label{c:trans-of-frob}
  If $f\colon A\to B$ is an algebra map between Frobenius algebras, then $\trans{f}\colon B\to A$ has the structure of a right $A$-module map.
\end{corollary}

\begin{proof}
Use \cref{o:add-geometric-realizations} to assume that $\cC$ has geometric realizations that commute with the tensor product in each variable.
By \cref{l:frob-is-mod-dual}, $\dual{B}$ is $B$ as a right $B$-module. 
It follows from \cref{l:dual-iff-bmod-dual} that as a right $A$-module, $\dual{B}$ is $B$ with the induced right $A$-action.
The result then follows from \cref{l:trans-is-rmod}.
\end{proof}

\subsection{Duality and adjunction}

Let $\cU$ be a symmetric monoidal 2-category.
An object $X\in \cU$ is called \emph{dualizable} if it is dualizable in the 1-core $X\in (\cU^{\simeq_1})^\dbl$.
We explicitly do \emph{not} consider the 2-categorical notion of \emph{full dualizability}, which requires the evaluation and coevaluation maps to admit adjoints.
We will, however, consider the interaction between duality and adjunction in $\cU$. 
For a rigid algebra $A$, the right adjoints of the unit and multiplication give $A$ the structure of a Frobenius algebra, exhibiting it as self-dual:

\begin{corollary}\label{c:rigid-is-frob}
  Let $A\in \Algrig(\cU)$ be a rigid algebra with unit $\eta$ and multiplication $\mu$.
  Then $\epsilon$ and $\delta$, the right adjoints of $\eta$ and $\mu$ respectively, equip $A$ with a Frobenius algebra structure.
\end{corollary}

\begin{proof}
  Using \cref{c:frob-equiv-explicit}, it only remains to check the counitality relations.
  These follow by adjunction from the unitality relations of $A$.
\end{proof}

Note that $\epsilon$ and $\delta$ are simultaneously the right adjoints and the transposes of $\eta$ and $\mu$. 
We wish to generalize this relationship to arbitrary maps of rigid algebras.
To that end, we work in the lax arrow category $\Funlax(\Delta^1, \cU)$, where objects are vertical arrows and morphisms are given by lax commuting squares (see \cref{e:lax-arrow-category}). In this setting, the interaction of duality and adjunction is more transparent. 
This is illustrated by the following result of Hoyois, Scherotzke, and Sibilla:

\begin{lemma}[{\cite[Lemma 2.4]{hoyois2017highertracesnoncommutativemotives}}]\label{l:dbl-arrow}
  An arrow $f\colon X\to Y\in \cU$ is dualizable as an object of $\Funlax(\Delta^1,\cU)$ if and only if $X$ and $Y$ are dualizable in $\cU$ and $f$ has a right adjoint, in which case the dual of $f$ is $\dual{f}=\trans{(f^R)}\colon \dual{X}\to \dual{Y}$.
\end{lemma}

If $f\colon X\to Y\in \cU$ is dualizable in $\Funlax(\Delta^1,\cU)$, then by \cref{l:dbl-arrow} $f$ has a right adjoint given by $f^R=\trans{(\dual{f})}$.
Alternatively, this can be phrased as an adjunction $\dual{f}\dashv\trans{f}$ in $\cU$.
The proof of \cref{l:dbl-arrow} also constructs the unit and counit of this adjunction, as we now describe.

Consider the trivially commuting squares
\[\begin{tikzcd}
  X & X && X & Y \\
  X & Y, && Y & Y
  \arrow[equals, from=1-1, to=1-2]
  \arrow[equals, from=1-1, to=2-1]
  \arrow["f", from=1-2, to=2-2]
  \arrow["f", from=1-4, to=1-5]
  \arrow["f"', from=1-4, to=2-4]
  \arrow[equals, from=1-5, to=2-5]
  \arrow["f"', from=2-1, to=2-2]
  \arrow[equals, from=2-4, to=2-5]
\end{tikzcd}\]
viewed as maps $\id_{X}\to f$ and $f\to \id_{Y}$ in $\Fun(\Delta^1,\cU)$, where the objects are vertical arrows and the morphisms go horizontally.
Assuming that $f$ is dualizable in $\Funlax(\Delta^1,\cU)$, taking the transpose of these maps produces lax commuting squares
\[\begin{tikzcd}
  {\dual{X}} & {\dual{X}} && {\dual{Y}} & {\dual{X}} \\
  {\dual{Y}} & {\dual{X},} && {\dual{Y}} & {\dual{Y}}.
  \arrow["{\dual{f}}"', from=1-1, to=2-1]
  \arrow[equals, from=1-2, to=1-1]
  \arrow[shorten <=9pt, shorten >=9pt, Rightarrow, from=1-2, to=2-1]
  \arrow[equals, from=1-2, to=2-2]
  \arrow["{\trans{f}}", from=1-4, to=1-5]
  \arrow[equals, from=1-4, to=2-4]
  \arrow[shorten <=9pt, shorten >=9pt, Rightarrow, from=1-5, to=2-4]
  \arrow["{\dual{f}}", from=1-5, to=2-5]
  \arrow["{\trans{f}}"', from=2-1, to=2-2]
  \arrow[equals, from=2-5, to=2-4]
\end{tikzcd}\]
The above 2-morphisms are respectively a unit and counit for an adjunction $\dual{f}\dashv\trans{f}$, and the zigzag identities follow from composing the two squares horizontally and vertically.

\begin{remark}
   In \cite{hoyois2017highertracesnoncommutativemotives}, the authors work with the model of the Gray product introduced in \cite{johnsonfreyd2017laxnatural}, rather than the one we adopt from \cite{gagna2021graytensorproducts}. 
   While a complete comparison of the coherence data in the two models has not been carried out, the basic operations such as gluing lax squares are compatible across both approaches.
\end{remark}

Let $A$ be an algebra in $\cU$ and let $f\colon M\to N$ be a map of left $A$-modules such that $f$ is dualizable in $\Funlax(\Delta^1,\cU)$. 
The transpose $\trans{f}\colon \dual{N}\to \dual{M}$ is a map of right $A$-modules by \cref{l:trans-is-rmod}, so the adjunction $\dual{f}\dashv \trans{f}$ induces the structure of an oplax linear right $A$-module map on $\dual{f}$ by an oplax version of \cref{p:projection-formula}.

We will require a more precise understanding of the interaction between the duality of $f$ and the resulting  oplax linear structure on $\dual{f}$.
To that end, we use \cref{c:modules-in-lax-arrow-category} to consider $f$ as a left $\id_A$-module in $\Fun(\Delta^1,\cU)$, and in particular $f\in \LMod_{\id_A}(\Funlax(\Delta^1,\cU))$.
By \cref{c:dual-of-lmod-is-rmod}, $\dual{f}\in \RMod_{\id_A}(\Funlax(\Delta^1,\cU))$, so using \cref{c:modules-in-lax-arrow-category} again we get that $\dual{f}$ is an oplax linear map of right $A$-modules.
The unit and counit of $\dual{f}\dashv\trans{f}$ constructed above can similarly be lifted to $\RMod_A^\oplax(\cU)$.

\begin{lemma}\label{l:two-oplax-structures}
  The adjunction $\dual{f}\dashv\trans{f}$ lifts to $\RMod^\oplax_A(\cU)$, where $\trans{f}$ is a map of right $A$-modules by \cref{l:trans-is-rmod} and $\dual{f}$ has the structure described in the above paragraph.
\end{lemma}

\begin{proof}
  Consider the maps $\id_M\to f$ and $f\to \id_N$ from above whose transposes produce the unit and counit of the adjunction $\dual{f}\dashv\trans{f}$.
  Using \cref{c:modules-in-lax-arrow-category}, these maps are morphisms in
  $$\Fun(\Delta^1,\LMod_A(\cU))\simeq \LMod_{\id_A}(\Fun(\Delta^1,\cU)),$$
  and in particular they are morphisms in $\LMod_{\id_A}(\Funlax(\Delta^1,\cU))$.
  By \cref{l:trans-is-rmod}, their transposes are in $\RMod_{\id_A}(\Funlax(\Delta^1,\cU))$, and in particular they are in 
  $$\RMod^\oplax_{\id_A}(\Funlax(\Delta^1,\cU))\simeq \Funlax(\Delta^1,\RMod^\oplax_A(\cU)).$$
  Thus, the unit and counit are 2-morphisms in $\RMod^\oplax_A(\cU)$.
\end{proof}

\subsection{Maps between rigid algebras}

In this subsection we study the 2-category of rigid algebras $\Algrig(\cU)$.
Our main result is that maps of rigid algebras are rigid algebras in the lax arrow category, from which we deduce that those maps are adjoint to their transpose and that $\Algrig(\cU)$ is a 1-category.

Suppose $f\colon A\to B\in \Alg(\cU)$ is a map of algebras.
By \cref{c:algebras-in-lax-arrow-category}, $f$ is an algebra object in the arrow category $\Fun(\Delta^1,\cU)$, and in particular in the lax arrow category $\Funlax(\Delta^1,\cU)$.

\begin{proposition}\label{p:rigid-map-is-rigid}
  Let $f\colon A\to B\in \Alg(\cU)$ be an algebra map.
  Then $f$ is rigid as an algebra in $\Funlax(\Delta^1,\cU)$ if and only if $A$ and $B$ are rigid.
\end{proposition}

\begin{proof}
  The restrictions to the source or target  $\Funlax(\Delta^1,\cU)\to\cU$ are symmetric monoidal, so the ``only if'' direction follows.
  For the ``if'' direction, the unit and multiplication of $f$ as an algebra in $\Fun(\Delta^1,A)$ are given by
  \[\begin{tikzcd}
    \one & A && {A\otimes A} & A \\
    \one & B, && {B\otimes B} & B.
    \arrow["{\eta_A}", from=1-1, to=1-2]
    \arrow["{\id_\one}"', from=1-1, to=2-1]
    \arrow["f", from=1-2, to=2-2]
    \arrow["{\mu_A}", from=1-4, to=1-5]
    \arrow["{f\otimes f}"', from=1-4, to=2-4]
    \arrow["f", from=1-5, to=2-5]
    \arrow["{\eta_B}"', from=2-1, to=2-2]
    \arrow["{\mu_B}"', from=2-4, to=2-5]
  \end{tikzcd}\]
  These squares have a right adjoint in $\Funlax(\Delta^1,\cU)$ by \cref{p:adj-lax-nat}, which is given by the Beck-Chevalley 2-morphism
  \[\begin{tikzcd}
    A & \one && A & {A\otimes A} \\
    B & \one, && {B} & {B\otimes B}.
    \arrow["{{\epsilon_A}}", from=1-1, to=1-2]
    \arrow["f"', from=1-1, to=2-1]
    \arrow["{{\epsilon_f}}"', shorten <=8pt, shorten >=8pt, Rightarrow, from=1-2, to=2-1]
    \arrow["{{\id_\one}}", from=1-2, to=2-2]
    \arrow["{{\delta_A}}", from=1-4, to=1-5]
    \arrow["f"', from=1-4, to=2-4]
    \arrow["{{\delta_f}}"', shorten <=8pt, shorten >=8pt, Rightarrow, from=1-5, to=2-4]
    \arrow["{{f\otimes f}}", from=1-5, to=2-5]
    \arrow["{{\epsilon_B}}"', from=2-1, to=2-2]
    \arrow["{{\delta_B}}"', from=2-4, to=2-5]
  \end{tikzcd}\]
  It remains to show that the second square is a map of $f$-bimodules, for which we use the projection formula of \cref{p:projection-formula}.
  For the projection formula, we need to show that a certain 2-morpshim is invertible in $\Funlax(\Delta^1,\cU)$.
  By \cite[Proposition 2.4.3]{abellan2024adjunctions}, it is enough to check that the 2-morphism is invertible when restricting to the source and target $\Funlax(\Delta^1,\cU)\to \cU\times\cU$, but there it follows from the projection formula for $\delta_A$ and $\delta_B$ respectively 
\end{proof}

An algebra map $f\colon A\to B$ endows $B$ with a right $A$-module structure for which $f$ is a map of right $A$-modules.
By \cref{c:trans-of-frob}, the same holds for its transpose $\trans{f}\colon B\to A$.

\begin{proposition}\label{p:rigid-map-adjoint}
  Let $f\colon A\to B\in \Algrig(\cU)$ be an algebra map between rigid algebras, then $\trans{f}$ is right adjoint to $f$ in $\RMod_A(\cU)$.
\end{proposition}

\begin{proof}
  \Cref{p:rigid-map-is-rigid,c:rigid-is-frob} imply that $f$ is a Frobenius algebra in $\Funlax(\Delta^1,\cU)$.
  \Cref{l:frob-is-mod-dual} then implies that $\dual{f}\simeq f$ as right $f$-modules (where we may use \cref{o:add-geometric-realizations} to assume the existence of geometric realizations that commute with the tensor product in each variable). 
  By restriction of scalars along the map of algebras $\id_A\to f$,
  it follows that $\dual{f}\simeq f$ as right $\id_A$-modules, which under the isomorphism of \cref{c:modules-in-lax-arrow-category}
  $$\RMod_{\id_A}(\Fun(\Delta^1,\cU))\simeq \Fun(\Delta^1,\RMod(\cU))$$
  implies that $\dual{f}\colon A\to B$ is $f$ with its natural structure as a map of right $A$-modules. 
  The result then follows from \cref{l:two-oplax-structures}, where we use the fact that $\RMod_A(\cU)\to \RMod_A^\oplax(\cU)$ is a locally full inclusion (\cref{l:funo-is-locally-full}).
\end{proof}

\begin{remark}
  An alternative approach to proving the adjunction in \cref{p:rigid-map-adjoint} is to verify the projection formula (\cref{p:projection-formula}).
  However, our chosen method has the advantage of directly showing that the resulting right $A$-module map structure on $\trans{f}$ coincides with the one obtained in \cref{c:trans-of-frob}.
\end{remark}

\begin{corollary}\label{c:rig-is-1-category}
  Every 2-morphism in $\Algrig(\cU)$ is invertible, i.e., $\Algrig(\cU)$ is a 1-category.
\end{corollary}

\begin{proof}
  Let $f,g\colon A\to B$ be maps in $\Algrig(\cU)$ and consider a 2-morphism $\alpha\colon f\Rightarrow g$.
  \Cref{p:rigid-map-is-rigid} implies that $f$ and $g$ are rigid algebras in $\Funlax(\Delta^1,\cU)$, hence $\alpha$ is an algebra map between rigid algebras.
  \Cref{p:rigid-map-adjoint} implies that $\alpha$ is left adjoint, so it follows from \cref{p:adj-lax-nat} that $\alpha$ is invertible.
\end{proof}

\subsection{Rigid commutative algebras}

In the second half of the paper we will consider rigidity in the context of commutative algebras.
Note that the map $\EE_1\to \EE_\infty$ induces a forgetful functor $(-)|_{\EE_1}\colon \CAlg(\cU)\to \Alg(\cU)$

\begin{definition}
  The category of \emph{rigid commutative algebras} is defined as the pullback
  $$\CAlgrig(\cU):=\CAlg(\cU)\times_{\Alg(\cU)} \Algrig(\cU).$$
  Thus, a commutative algebra $A\in \CAlg(\cU)$ is rigid if the underlying algebra $A|_{\EE_1}$ is rigid.
\end{definition}

Note that if $A$ is a commutative algebra in $\cU$, then there is a canonical isomorphism between $A|_{\EE_1}$ and $A|_{\EE_1}^\rev$, where the second has its multiplication reversed.
This in turn induces an isomorphism between left and right modules
$$\LMod_{A|_{\EE_1}}(\cU)\simeq \RMod_{A|_{\EE_1}}(\cU),$$
so we denote both sides by $\Mod_A(\cU)$, simply called \emph{$A$-modules}.
We thus translate \cref{p:rigid-map-adjoint} to the commutative setting:

\begin{corollary}\label{c:com-rigid-map-adjoint}
  Let $f\colon A\to B\in \CAlgrig(\cU)$ be a commutative algebra map between rigid commutative algebras, then $\trans{f}$ is right adjoint to $f$ in $\Mod_A(\cU)$.
\end{corollary}

\begin{lemma}\label{l:forgetful-is-2conservative}
  The forgetful functors 
  \begin{equation*}
    (-)|_{\EE_1}\colon \CAlg(\cU)\to\Alg(\cU),\qquad
    (-)|_{\EE_1}\colon \CAlg^\lax(\cU)\to\Alg^\lax(\cU)
  \end{equation*}
  are conservative on 1-morphisms and 2-morphisms.
\end{lemma}

\begin{proof}
  The fact that $\CAlg(\cU)\to\Alg(\cU)$ is conservative on 1-morphisms follows from passing to the 1-core. 
  This also implies that $\CAlg^\lax(\cU)\to\Alg^\lax(\cU)$ is conservative on 1-morphisms:
  if a lax map of commutative algebras is invertible in $\Alg^\lax(\cU)$, then it is in particular strong, as it is enough to verify on the multiplication and unit.

  A 2-morphism in $\CAlg^\lax(\cU)$ gives a 1-morphism in the oplax arrow category, where by \cref{c:algebras-in-lax-arrow-category} we have an equivalence
  $$\Fun^\oplax(\Delta^1,\CAlg^\lax(\cU))\simeq \CAlg^\oplax(\Fun^\lax(\Delta^1,\cU)).$$
  The same holds for a 2-morphism in $\Alg^\lax(\cU)$. Thus, the fact that $\CAlg^\lax(\cU)\to\Alg^\lax(\cU)$ is conservative on 1-morphisms implies the result for 2-morphisms.
  Finally, it follows that $\CAlg(\cU)\to\Alg(\cU)$ is conservative on 2-morphisms, as $\CAlg(\cU)\to \CAlg^\lax(\cU)$ is a locally full inclusion (\cref{l:funo-is-locally-full}).
\end{proof}

\begin{corollary}\label{c:com-rig-is-1-category}
  Every 2-morphism in $\CAlgrig(\cU)$ is invertible, i.e., $\CAlgrig(\cU)$ is a 1-category.
\end{corollary}

\begin{proof}
  Follows from \cref{c:rig-is-1-category,l:forgetful-is-2conservative}.
\end{proof}

\section{Cospans}\label{s:cospans}

Using our extended definition, we can consider rigid algebras in 2-categories which are not of the form ``categories-of-categories''.
For the remainder of the paper, we focus our attention on the 2-category of cospans.
In this section we study cospans and rigid commutative algebras therein, and in the next section we will see that this example is in a certain sense universal.

Fix a 1-category $\cC$ with finite colimits.
For the definition of the 2-category of cospans in $\cC$, which we denote $\coSpan(\cC)$, we follow \cite{gaitsgory2019study}, \cite{stefanich2020highersheaftheoryi} and \cite{Macpherson2020ABY}.
The lower dimensional cells of $\coSpan(\cC)$ are described as follows:
\begin{itemize}
  \item The objects are the same as in $\cC$.
  \item The morphisms  
  $\begin{tikzcd}[column sep=small]
    A & B
    \arrow[dotted, from=1-1, to=1-2]
  \end{tikzcd}$ 
  are cospans in $\cC$
  \[\begin{tikzcd}
  & X  &   \\
  A\arrow[ru] & & B,\arrow[lu]
  \end{tikzcd}\]
  with composition given by pushouts
  \[\begin{tikzcd}
    && {X\sqcup_B Y} \\
    & X && Y \\
    A && B && C.
    \arrow["\lrcorner"{anchor=center, pos=0.125, rotate=-45}, draw=none, from=1-3, to=3-3]
    \arrow[from=2-2, to=1-3]
    \arrow[from=2-4, to=1-3]
    \arrow[from=3-1, to=2-2]
    \arrow[from=3-3, to=2-2]
    \arrow[from=3-3, to=2-4]
    \arrow[from=3-5, to=2-4]
  \end{tikzcd}\]
  \item The 2-morphisms are maps between cospans
  \[\begin{tikzcd}
    & X\arrow[dashed,dd] &   \\
    A\arrow[ru]\arrow[rd] & & B\arrow[lu]\arrow[ld]\\
    & Y. &
  \end{tikzcd}\]
\end{itemize}
Moreover, $\coSpan(\cC)$ has a symmetric monoidal structure given by coproducts in $\cC$, with unit the initial object $\emptyset\in \cC$.
Note that the coproducts and initial object in $\cC$ are not the coproducts and initial object in $\coSpan(\cC)$.

\begin{remark}
  It is more common to work with the span 2-category $\Span(\cC)$, which has the same objects, the morphisms are spans, and the 2-morphisms are maps between spans.
  The definitions of $\Span(\cC)$ and $\coSpan(\cC)$ are related by
  $$\coSpan(\cC)=\Span(\cC^{\op})^{\op_2}.$$
\end{remark}

\begin{remark}
  There is another variant of the 2-category of cospans, called 2-cospans, where the 2-morphisms are themselves cospans. 
  The version we consider is sometimes called 1.5-cospans, as we only take one half of the 2-morphisms.
\end{remark}

Consider the functor $\cC\to \coSpan(\cC)$ defined in e.g.\ \cite{Macpherson2020ABY}, which we call the \emph{right-way functor}, that acts as the identity on objects and sends a map $f\colon A\to B\in \cC$ to the right-way map 
\[\begin{tikzcd}
    & B  &   \\
    A\arrow["f",ru] & & B\arrow[lu,equal].
\end{tikzcd}\]
The right-way functor is symmetric monoidal with the cocartesian structure on $\cC$, so in particular it factors through commutative algebras 
$$\cC\simeq \CAlg(\cC)\to \CAlg(\coSpan(\cC)).$$ 
This gives every $A\in \cC$ a canonical commutative algebra structure in $\coSpan(\cC)$, with multiplication given by the fold map and unit given by the initial map:
\[\begin{tikzcd}
	& A &&&& A \\
	{A\sqcup A} && A && \emptyset && {A}.
	\arrow[from=2-1, to=1-2, "\nabla"]
	\arrow[from=2-3, to=1-2,equal]
	\arrow[from=2-5, to=1-6]
	\arrow[from=2-7, to=1-6, equal]
\end{tikzcd}\]
Our goal in this section is to prove that this canonical commutative algebra structure is rigid, and moreover that the right-way functor induces an isomorphism
$$\cC\isoto\CAlgrig(\coSpan(\cC)).$$

\subsection{The universal property of cospans}

The 2-category $\coSpan(\cC)$ has a universal property described in \cite{gaitsgory2019study,stefanich2020highersheaftheoryi,Macpherson2020ABY}.
As part of this property, it is implied that the right-way maps are left adjoints in $\coSpan(\cC)$.
We will also require the converse implication:

\begin{lemma}\label{l:ladj_cospan}
  A cospan $A\to X\leftarrow B$ is left adjoint in  $\coSpan(\cC)$ if and only if the wrong-way component is invertible $X\xleftarrow{\sim} B$, i.e., it is equivalent to a right-way map.
\end{lemma}

\begin{proof}
  For the ``if'' direction, let $A\xrightarrow{f} B= B$ be a right-way map.
  Its right adjoint is the mirror image wrong-way map $B=B\xleftarrow{f} A$, the unit of the adjunction is given by taking $f$ as a 2-morphism, and the counit is given by the fold map:
  \[\begin{tikzcd}
    && A \\
    && B \\
    & B && B \\
    A && B && A
    \arrow["f", dashed, from=1-3, to=2-3]
    \arrow["\lrcorner"{anchor=center, pos=0.125, rotate=-45}, draw=none, from=2-3, to=4-3]
    \arrow[equals, from=3-2, to=2-3]
    \arrow[equals, from=3-4, to=2-3]
    \arrow[curve={height=-18pt}, equals, from=4-1, to=1-3]
    \arrow["f", from=4-1, to=3-2]
    \arrow[equals, from=4-3, to=3-2]
    \arrow[equals, from=4-3, to=3-4]
    \arrow[curve={height=18pt}, equals, from=4-5, to=1-3]
    \arrow["f"', from=4-5, to=3-4]
  \end{tikzcd}
  \hspace{15pt}
  \begin{tikzcd}
    && B \\
    && {B\sqcup_A B} \\
    & B && B \\
    B && A && {B.}
    \arrow[dashed, from=2-3, to=1-3]
    \arrow["\lrcorner"{anchor=center, pos=0.125, rotate=-45}, draw=none, from=2-3, to=4-3]
    \arrow[from=3-2, to=2-3]
    \arrow[from=3-4, to=2-3]
    \arrow[curve={height=-18pt}, equals, from=4-1, to=1-3]
    \arrow[equals, from=4-1, to=3-2]
    \arrow["f"', from=4-3, to=3-2]
    \arrow["f", from=4-3, to=3-4]
    \arrow[curve={height=18pt}, equals, from=4-5, to=1-3]
    \arrow[equals, from=4-5, to=3-4]
  \end{tikzcd}\]
  The verification of the zigzag identities is straightforward, see e.g.\ \cite[Proposition 3.3.1]{stefanich2020highersheaftheoryi} for details.

  For the ``only if'' direction, suppose that the cospan on the left has a right adjoint given by the cospan on the right:
  \[\begin{tikzcd}[sep=scriptsize]
    & X &&&& Y \\
    A && B && B && {A.}
    \arrow["f", from=2-1, to=1-2]
    \arrow[shorten <=14pt, shorten >=14pt, dotted, from=2-1, to=2-3]
    \arrow["g"', from=2-3, to=1-2]
    \arrow["h", from=2-5, to=1-6]
    \arrow[shorten <=14pt, shorten >=14pt, dotted, from=2-5, to=2-7]
    \arrow["k"', from=2-7, to=1-6]
  \end{tikzcd}\]
  The counit is a map $Y\sqcup_A X\to B$ such that the following commutes:
  \[\begin{tikzcd}
	&& B \\
	&& {Y\sqcup_A X} \\
	& Y && X \\
	B && A && {B.}
	\arrow[from=2-3, to=1-3,dashed]
	\arrow["\lrcorner"{anchor=center, pos=0.125, rotate=-45}, draw=none, from=2-3, to=4-3]
	\arrow[from=3-2, to=2-3]
	\arrow[from=3-4, to=2-3]
	\arrow[curve={height=-18pt}, equals, from=4-1, to=1-3]
	\arrow["k", from=4-1, to=3-2]
	\arrow["h"', from=4-3, to=3-2]
	\arrow["f", from=4-3, to=3-4]
	\arrow[curve={height=18pt}, equals, from=4-5, to=1-3]
	\arrow["g"', from=4-5, to=3-4]
\end{tikzcd}\]
  Consider the compositions $r_k\colon Y\to Y\sqcup_A X\to B$ and $r_g\colon X\to  Y\sqcup_A X\to B$ from the diagram above; it follows that $r_k$ and $r_g$ are retracts of $k$ and $g$ respectively.
  As a map from the pushout, the counit is determined by $r_k$ and $r_g$ (and commutation data which we keep implicit), so we denote it by $[r_k,r_g]\colon Y\sqcup_A X\to B$.

  The unit is a map $A\to X\sqcup_B Y$ with a similar commuting diagram, and the zigzag identity tells us that the composition 
  $$X\to X\sqcup_B Y\sqcup_A X\to X$$
  is isomorphic to the identity.
  Explicitly, the first map is the inclusion to the third coordinate
  $i_3\colon X\to X\sqcup_B Y\sqcup_A X$, and the second map is given by $[\id_X,gr_k,gr_g]\colon X\sqcup_B Y\sqcup_A X\to X$.
  Their composition is then given by  
  $$[\id_X,gr_k,gr_g]\circ i_3\simeq gr_g.$$
  Knowing that the composition is isomorphic to the identity, we conclude that $r_g$ is also a section of $g$, hence $g$ is invertible.
\end{proof}

It follows from \cref{l:ladj_cospan} that the right-way functor $\cC\to \coSpan(\cC)$ sends every map in $\cC$ to a left adjoint map in $\coSpan(\cC)$.
In fact, the right-way functor satisfies a stronger condition called \emph{cobase change}:

\begin{definition}
  Let $\cU\in \Cat_2$ be a 2-category. 
  A functor $F:\cC\to \cU$ is said to satisfy \emph{cobase change} (or the co-Beck-Chevalley property) if the following conditions hold:
  \begin{enumerate}
    \item For every map $f:A\to B\in \cC$, its image $F(f):F(A)\to F(B)$ has a right adjoint in $\cU$.
    \item For every pushout square in $\cC$
    \[\begin{tikzcd}
        A & B \\
        C & P,
        \arrow["f", from=1-1, to=1-2]
        \arrow["g"', from=1-1, to=2-1]
        \arrow["{g'}", from=1-2, to=2-2]
        \arrow["{f'}"', from=2-1, to=2-2]
        \arrow["\lrcorner"{anchor=center, pos=0.125, rotate=180}, draw=none, from=2-2, to=1-1]
    \end{tikzcd}\]
    with image in $\cU$ given by
    \[\begin{tikzcd}
      {F(A)} & {F(B)} \\
      {F(C)} & {F(P),}
      \arrow["{F(f)}", from=1-1, to=1-2]
      \arrow["{F(g)}"', from=1-1, to=2-1]
      \arrow["{F(g')}", from=1-2, to=2-2]
      \arrow["{F(f')}"', from=2-1, to=2-2]
    \end{tikzcd}\]
    the corresponding Beck-Chevalley 2-morphism $F(g)F(f)^R\Rightarrow F(f')^RF(g')$ is invertible.
  \end{enumerate}
\end{definition}

The universal property of cospans states that the right-way functor $\cC\to \coSpan(\cC)$ is the initial functor from $\cC$ satisfying cobase change (\cite[Theorem 4.2.6]{stefanich2020highersheaftheoryi}).
This also lifts to the symmetric monoidal setting.

\begin{proposition}[{\cite[Corollary 1.2.2]{stefanich2020highersheaftheoryi}, \cite[Theorem 4.4.6]{Macpherson2020ABY}}]\label{p:universal-propertey-of-cospans}
  Let $\cU\in \Catsm_2$ be a symmetric monoidal 2-category.
  Precomposition with the right-way functor $\cC\to \coSpan(\cC)$ induces an isomorphism between $\Mapo(\coSpan(\cC),\cU)$ and the subspace of $\Mapo(\cC,\cU)$ consisting of functors satisfying cobase change.
\end{proposition}

\subsection{Rigid commutative algebras in cospans}
In this subsection, we prove that the right-way functor induces an isomorphism
$$\cC\isoto \CAlgrig(\coSpan(\cC)).$$
The first step is to show that the right-way functor $\cC\to \CAlg(\coSpan(\cC))$ lands in rigid commutative algebras.

\begin{lemma}\label{l:can-is-rig}
  The canonical commutative algebra structure on every $A\in \coSpan(\cC)$ is rigid.
\end{lemma}

\begin{proof}
  The canonical commutative algebra structure on $A\in\coSpan(\cC)$ comes via the right-way functor, so by \cref{l:ladj_cospan} the multiplication and unit maps are left adjoints.
  It remains to check that the multiplication map is left adjoint in $A$-bimodules.

  Recall that the multiplication map is the image under the right-way functor of the fold map $\nabla\colon A\sqcup A\to A$.
  To show that the multiplication map satisfies the projection formula,
  consider the following pushout square in $\cC$:
  \[\begin{tikzcd}[column sep=large]
    {A\sqcup A\sqcup A\sqcup A} & {A\sqcup A\sqcup A} \\
    {A\sqcup A} & A
    \arrow["{\id\sqcup \nabla \sqcup \id}", from=1-1, to=1-2]
    \arrow["{\nabla\sqcup \nabla}"', from=1-1, to=2-1]
    \arrow["{\nabla_3}", from=1-2, to=2-2]
    \arrow["\nabla"', from=2-1, to=2-2]
    \arrow["\lrcorner"{anchor=center, pos=0.125, rotate=180}, draw=none, from=2-2, to=1-1]
  \end{tikzcd}\]
  and its image in $\coSpan(\cC)$ under the right-way functor.
  By \cref{p:projection-formula}, it suffices to check that the resulting commuting square in $\coSpan(\cC)$ has an invertible Beck-Chevalley 2-morphisms.
  This follows from the cobase change property of the right-way functor.
\end{proof}

Denote by $\coSpan_1(\cC):=\coSpan(\cC)^{\simeq_1}$ the 1-category of cospans, originally defined in \cite{barwick2014spectralmackeyfunctorsequivariant}.
The right-way functor lands in the 1-core $\cC\to \coSpan_1(\cC)$, and is symmetric monoidal with respect to the cocartesian structure on $\cC$, so it factors through commutative algebras $\cC\to \CAlg(\coSpan(\cC))$.

\begin{remark}\label{r:cospan-can-is-frob}
  The image of $A\in \cC$ under $\cC\to \CAlg(\coSpan_1(\cC))$ can be promoted canonically (but not uniquely, see \cite{Contreras_2022}) to a Frobenius algebra via the counit 
  \[\begin{tikzcd}
    & A  &   \\
    A\arrow[ru,equal] & & \emptyset\arrow[lu].
    \end{tikzcd}\] 
  Note that this counit is the mirror image of the unit; the resulting self-duality of any $A,B\in\coSpan_1(\cC)$ is such that the transpose of a cospan $A\to X\leftarrow B$ is the mirror image cospan $B\to X\leftarrow A$.
  By moving to the 2-categorical setting, we realize this canonical Frobenius algebra structure as arising   from rigidity.
  In particular, note that every right-way cospan $A\to B = B$ is a map of commutative algebras, so by \cref{c:com-rigid-map-adjoint} it is left adjoint to its transpose, and indeed both the right adjoint and the  transpose are given by the mirror image.
\end{remark}

\begin{lemma}\label{l:1-cat-right-way-fully-faithful}
  The 1-categorical right-way functor $\cC\to \CAlg(\coSpan_1(\cC))$ is fully faithful, with 
  essential image those $A\in \CAlg(\coSpan_1(\cC))$ whose unit map $\emptyset\to X\leftarrow A$ is a right-way map, i.e., $X\leftarrow A$ is an isomorphism.
\end{lemma}

\begin{proof}
  There is a factorization system on $\coSpan_1(\cC)$ given by the right-way and wrong-way maps (\cite[Proposition 4.9]{haugseng2023twovariablefibrationsfactorisationsystems}).
  This induces a factorization system on $\coSpan(\cC)_{\emptyset/}$,
  so by \cite[Lemma 3.1.19]{anel2022leftexactlocalizations}, $\cC\to \coSpan_1(\cC)_{\emptyset/}$ is fully faithful with essential image those cospans $\emptyset \to X\leftarrow A$ where $X\leftarrow A$ is an isomorphism.

  As $\emptyset\in \coSpan(\cC)$ is the unit, $\coSpan(\cC)_{\emptyset/}$ inherits a symmetric monoidal strucure and $\cC\hookrightarrow \coSpan(\cC)_{\emptyset/}$ is symmetric monoidal with the cocartesian structure on $\cC$.
  Taking $\CAlg$ on both sides, we get
  $$\cC\simeq \CAlg(\cC)\hookrightarrow \CAlg(\coSpan_1(\cC)_{\emptyset/})\simeq \CAlg(\coSpan_1(\cC)),$$
  where the last isomorphism promotes $A\in \CAlg(\coSpan_1(\cC))$ to an object of $\coSpan_1(\cC)_{\emptyset/}$ using the unit map.
  
\end{proof}

\begin{proposition} \label{p:can-equiv-rig}
  The right-way functor $\cC\to \CAlg(\coSpan(\cC))$ factors as an isomorphism through rigid commutative algebras $\cC\isoto \CAlgrig(\coSpan(\cC))$.
\end{proposition}

\begin{proof}
  The right-way functor $\cC\to \CAlg(\coSpan(\cC))$ factors through the full subcategory of rigid commutative algebras by \cref{l:can-is-rig}.
  We want to show that $\cC\to \CAlgrig(\coSpan(\cC))$ is fully faithful and essentially surjective. 

  By \cref{c:rig-is-1-category}, $\CAlgrig(\coSpan(\cC))$ is a 1-category, so it is a full subcategory of the 1-core $\CAlgrig(\coSpan(\cC))\subseteq \CAlg(\coSpan_1(\cC))$.
  It then follows from \cref{l:1-cat-right-way-fully-faithful} that $\cC\to \CAlgrig(\coSpan(\cC))$ is fully faithful.  
  To see that it is essentially surjective, we need to show that every rigid commutative algebra comes from $\cC$ via the right-way functor. 
  The essential image of the right-way functor was described in \cref{l:1-cat-right-way-fully-faithful} as those commutative algebras whose unit is a right-way map.
  The unit of a rigid commutative algebra is indeed a right-way map, as it is a left adjoint (\cref{l:ladj_cospan}).
\end{proof}

\section{The adjunction of cospans and rigid commutative algebras}\label{s:the-adjunction}

Our goal in this section is to establish a new universal property for cospans, formulated using rigid commutative algebras.
This will take the form of an adjunction $\coSpan\dashv \CAlgrig$.
The unit of this adjunction will be the isomorphism $\cC\isoto \CAlgrig(\coSpan(\cC))$ of \cref{p:can-equiv-rig}, and the counit $\coSpan(\CAlgrig(\cU))\to \cU$ will come from the forgetful functor $\CAlgrig(\cU)\to \cU$ via the standard universal property of cospans.

We cannot yet phrase this adjunction formally, because the source and target of the two constructions are incompatible.
The functor
$$\coSpan\colon \Cat_1^\rex\to \Catsm_2$$
has as source 1-categories with finite colimits and right exact functors, while the functor
$$\CAlgrig\colon \Catsm_2\to \Cat_1$$
lands in 1-categories (\cref{c:com-rig-is-1-category}) not necessarily with finite colimits.
Moreover, even if it happens that $\CAlgrig(\cU)$ has finite colimits for some $\cU\in \Catsm_2$, it is not clear that the forgetful functor $\CAlgrig(\cU)\to \cU$ has cobase change in order to apply the universal property.

To resolve this mismatch, we will show that if $\CAlgrig(\cU)$ has pushouts which are given by relative tensor products, then $\CAlgrig(\cU)\to \cU$ has cobase change.
This condition is satisfied by $\coSpan(\cC)$ for every $\cC\in \Cat_1^\rex$, so we can restrict to the locally full subcategory $\Catrigbar\subseteq\Catsm_2$ where this condition is satisfied and prove the adjunction there:
$$\coSpan\colon \Cat_1^{\rex}\rightleftarrows\Catrigbar\colon \CAlgrig.$$

\subsection{Relative tensor products}

Let $\cU$ be a 2-category, $I$ a 1-category, and $D\colon I\to \cU$ a functor. 
A \emph{1-colimit} of $D$ is a colimit in the 1-core $D\colon I\to \cU^{\simeq_1}$.
A \emph{2-colimit} of $D$ is an enriched colimit as in \cite{hinich2023enrichedcolimits}.
Explicitly, a cocone $X\in \cU_{D/}$ is a 2-colimit if for every $Y\in \cU$ there is an equivalence of categories
$$\hom_\cU(X,Y)\isoto \lim_{i\in I^\op}\hom_\cU(D(i),Y).$$

\begin{remark}
  A 2-colimit, if it exists, is also a 1-colimit, as the core functor commutes with limits. 
  The converse need not hold: $\pt\in \mathrm{B}^2\NN$ is 1-initial but not 2-initial.
\end{remark}

For the remainder of this subsection let $\cU$ be a symmetric monoidal 2-category and fix a commutative algebra $A\in \CAlg(\cU)$.

\begin{definition}[{\cite[Construction 4.4.2.7]{ha}}]\label{d:bar-construction}
  For $M,N\in \Mod_A(\cU)$, the \emph{bar complex} of $M$ and $N$ over $A$ is the simplicial diagram in $\cU$
  \[\Br_A(M,N):=\begin{tikzcd}
    {\cdots M\otimes A \otimes A\otimes N} & {M\otimes A\otimes N} & {M\otimes N}
    \arrow[shift right=4, from=1-1, to=1-2]
    \arrow[from=1-1, to=1-2]
    \arrow[shift left=4, from=1-1, to=1-2]
    \arrow[shift right=2, from=1-2, to=1-1]
    \arrow[shift left=2, from=1-2, to=1-1]
    \arrow[shift left=2, from=1-2, to=1-3]
    \arrow[shift right=2, from=1-2, to=1-3]
    \arrow[from=1-3, to=1-2]
  \end{tikzcd}\]
  where the face maps are given either by the multiplication of $A$ or by the action of $A$ on $M$ and $N$, and the degeneracy maps are given by the unit of $A$.
  The \emph{relative tensor product} $M\otimes_A N$ is defined as the 2-colimit of this bar complex, when it exists.
\end{definition}

\begin{lemma}\label{l:adjunction-bar-complexes}
  Let $M_0,M_1,N\in \Mod_A(\cU)$, and suppose that $f\colon M_0\to M_1$ is a map of $A$-modules which has a right adjoint in $\Mod_A(\cU)$. 
  Then the induced map on bar complexes
  $$\Br_A(f,N)\colon \Br_A(M_0,N)\to \Br_A(M_1,N)$$
  has a right adjoint in $\Fun(\mathbf{\Delta}^\op,\cU)$.
\end{lemma}

\begin{proof}
  Let $\alpha\colon [n]\to [m]\in \mathbf{\Delta}$, and consider the naturality square induced from $\alpha$:
  \[\begin{tikzcd}
    {M_0\otimes A^{\otimes m}\otimes N} & {M_1\otimes A^{\otimes m}\otimes N} \\
    {M_0\otimes A^{\otimes n}\otimes N} & {M_1\otimes A^{\otimes n}\otimes N}.
    \arrow["f", from=1-1, to=1-2]
    \arrow[from=1-1, to=2-1]
    \arrow[from=1-2, to=2-2]
    \arrow["f"', from=2-1, to=2-2]
  \end{tikzcd}\]
  By \cref{c:adj-nat}, it is enough to show that the corresponding Beck-Chevalley 2-morphism 
  \[\begin{tikzcd}
    {M_1\otimes A^{\otimes m}\otimes N} & {M_0\otimes A^{\otimes m}\otimes N} \\
    {M_1\otimes A^{\otimes n}\otimes N} & {M_0\otimes A^{\otimes n}\otimes N}
    \arrow["{f^R}", from=1-1, to=1-2]
    \arrow[from=1-1, to=2-1]
    \arrow[shorten <=14pt, shorten >=14pt, Rightarrow, from=1-2, to=2-1]
    \arrow[from=1-2, to=2-2]
    \arrow["{f^R}"', from=2-1, to=2-2]
  \end{tikzcd}\]
  is invertible.
  Every morphism in $\mathbf{\Delta}$ is a composition of the standard faces and degeneracies, so it is enough to check those.
  Of them, the only non-trivial case is the face map $\delta_0\colon [n]\to [n+1]$, in which case the above lax commuting square is induced from the lax linear map of $A$-modules structure which $f^R$ has by \cref{p:projection-formula}:
  \[\begin{tikzcd}
    {M_1\otimes A} & {M_0\otimes A} \\
    {M_1} & {M_0.}
    \arrow["f", from=1-1, to=1-2]
    \arrow[from=1-1, to=2-1]
    \arrow[shorten <=10pt, shorten >=10pt, Rightarrow, from=1-2, to=2-1]
    \arrow[from=1-2, to=2-2]
    \arrow["f"', from=2-1, to=2-2]
  \end{tikzcd}\]
  The fact that the adjunction $f\dashv f^R$ is in $\Mod_A(\cU)$ implies that the above square commutes strongly.
\end{proof}

\begin{corollary}\label{c:relative-tensor-left-adjoint}
  In the setting of \cref{l:adjunction-bar-complexes}, assume further that the relative tensor products $M_0\otimes_A N$ and $M_1\otimes_A N$ exist.
  Then 
  $$f\otimes_A N\colon M_0\otimes_A N\to M_1\otimes_A N$$
  has a right adjoint in $\cU$.
\end{corollary}

\begin{proof}
  The 2-colimit is a (partially defined) 2-functor, so it preserves adjunctions. 
  The result then follows from \cref{l:adjunction-bar-complexes}.
\end{proof}

\begin{lemma}\label{l:beck-chevalley-bar-compatible}
  Let $f\colon M_0\to M_1$ and $g\colon N_0\to N_1$ be $A$-module maps such that the relative tensor products $M_i\otimes_A N_j$ exist.
  Assume that $f$ has a right adjoint in $\Mod_A(\cU)$, then for the commuting square in $\cU$
  \[\begin{tikzcd}
    {M_0\otimes_A N_0} & {M_1\otimes_A N_0} \\
    {M_0\otimes_A N_1} & {M_1\otimes_A N_1.}
    \arrow["{f\otimes_A N_0}", from=1-1, to=1-2]
    \arrow["{M_0\otimes_A g}"', from=1-1, to=2-1]
    \arrow["{M_1\otimes_A g}", from=1-2, to=2-2]
    \arrow["{f\otimes_A N_1}"', from=2-1, to=2-2]
  \end{tikzcd}\]
  the corresponding Beck-Chevalley 2-morphism  is an isomorphism.
\end{lemma}

\begin{proof}
  The above square is the 2-colimit of the square
  \[\begin{tikzcd}[column sep=huge]
    {\Br_A(M_0,N_0)} & {\Br_A(M_1,N_0)} \\
    {\Br_A(M_0,N_1)} & {\Br_A(M_1,N_1)}
    \arrow["{\Br_A(f,N_0)}", from=1-1, to=1-2]
    \arrow[from=1-1, to=2-1]
    \arrow[from=1-2, to=2-2]
    \arrow["{\Br_A(f,N_1)}"', from=2-1, to=2-2]
  \end{tikzcd}\quad\in\quad  \Fun(\mathbf{\Delta}^\op,\cU),\]
  so it is enough to show that the Beck-Chevalley 2-morphism is an isomorphism here. Note that, as $f$ has a right adjoint in $\Mod_A(\cU)$, \cref{l:adjunction-bar-complexes} implies that $\Br_A(f,N_j)$ has a right adjoint.
  To check that the Beck-Chevalley 2-morphism is invertible in $\Fun(\mathbf{\Delta}^\op,\cU)$, it is enough to check level-wise
  \[\begin{tikzcd}
    {M_0\otimes A^{\otimes n}\otimes N_0} & {M_1\otimes A^{\otimes n}\otimes N_0} \\
    {M_0\otimes A^{\otimes n}\otimes N_1} & {M_1\otimes A^{\otimes n}\otimes N_1}.
    \arrow[from=1-1, to=1-2]
    \arrow[from=1-1, to=2-1]
    \arrow[from=1-2, to=2-2]
    \arrow[from=2-1, to=2-2]
  \end{tikzcd}\]
  This square decomposes as a tensor product of independent squares, so it suffices to consider the component involving $f$:
  \[\begin{tikzcd}
    {M_0} & {M_1} \\
    {M_0} & {M_2}.
    \arrow["f", from=1-1, to=1-2]
    \arrow[equals, from=1-1, to=2-1]
    \arrow[equals, from=1-2, to=2-2]
    \arrow["f", from=2-1, to=2-2]
  \end{tikzcd}\] 
  Since this square is degenerate, the corresponding Beck-Chevalley 2-morphism is invertible.
\end{proof}

Consider $B,C\in \CAlg(\cU)_{A/}$.
If $\cU$ has all geometric realizations and the tensor product commutes with them in each variable, then the relative tensor product $B\otimes_A C$ has the structure of a commutative algebra which is the pushout of $B$ and $C$ over $A$ in $\CAlg(\cU)$.
However, the assumption that all geometric realizations exist is too restrictive for us.
Importantly, the category of cospans does not necessarily have all geometric realizations.
This leads us to consider a relaxed condition, where the relative tensor product exists only for some specific modules.
In the language of \cite{dancohen2024relative}, these modules are called \emph{bar compatible}.

\begin{definition}\label{d:cfbar}
  Let $M,N\in \Mod_A(\cU)$. We say that $M$ and $N$ are \emph{bar compatible} over $A$ if for every $X,Y\in \cU$ the following conditions hold:
  \begin{enumerate}
    \item The relative tensor product $(X\otimes M)\otimes_A(N\otimes Y)$ exists.
    \item The map 
    $$(X\otimes M)\otimes_A(N\otimes Y)\to X\otimes (M\otimes_A N)\otimes Y$$
    given by the assembly map
    $$|\Br_A(X\otimes M,N\otimes Y)|
    \to X\otimes|\Br_A(M,N)|\otimes Y$$
    is an isomorphism.
  \end{enumerate}
\end{definition}

\begin{lemma} \label{p:bar-construction-is-pushout}
  Let $B,C\in \CAlg(\cU)_{A/}$ such that $B$ and $C$ are bar compatible over $A$.
  Then $B$ and $C$ have a 1-pushout over $A$ in $\CAlg(\cU)$, whose underlying object in $\cU$ is $B\otimes_A C$.
\end{lemma}

\begin{proof}
  As we only require a 1-pushout, we may pass to the 1-core and assume that $\cU$ is a 1-category.
  Generally, pushouts can be calculated by coproducts and geometric realizations.
  In $\CAlg(\cU)$, coproducts are given by tensor products \cite[Proposition 3.2.4.7]{ha}, hence the pushout is given by the colimit of $\Br_A(B,C)$ in $\CAlg(\cU)$.
  This colimit exists, and is preserved by the forgetful functor $\CAlg(\cU)\to \cU$, by \cite[Corollary 3.2.3.2]{ha}.

  Note that \cite[Corollary 3.2.3.2]{ha} assumes the existence of \emph{all} geometric realizations, and that they all commute with tensor products in each variable. 
  It can be checked that the proof only needs $B$ and $C$ to be bar compatible over $A$.
  Alternatively, we could move to a conservative cocompletion where we freely add all missing colimits, and apply the claim there.
  This conservative cocompletion is the localization of $\PSh(\cU)$ identifying the formal colimit of $\Br_A(X\otimes B,C\otimes Y)$ with $(X\otimes B)\otimes_A (C\otimes Y)$; the second condition in \cref{d:cfbar} implies that this localization is symmetric monoidal.
\end{proof}

\begin{remark}
  We expect the 1-pushout in \cref{p:bar-construction-is-pushout} to also always be a 2-pushout. 
\end{remark}

\subsection{Rigid bar compatibility}

We now consider symmetric monoidal 2-categories whose  rigid commutative algebras are bar compatible.

\begin{definition}\label{d:rigid-bar-compatible}
  A symmetric monoidal 2-category $\cU\in \Catsm_2$ is called \emph{rigid bar compatible} if for every $A\in \CAlgrig(\cU)$ and $B,C\in \CAlgrig(\cU)_{A/}$ the following holds:
  \begin{enumerate}
    \item\label{i:relative-tensor-of-rigid-exist} $B$ and $C$ are bar compatible over $A$, and
    \item\label{i:relative-tensor-of-rigid-is-rigid} $B\otimes_A C$ is rigid (where we consider $B\otimes_A C\in \CAlg(\cU)$ by \cref{p:bar-construction-is-pushout}).
  \end{enumerate}
  A symmetric monoidal functor $F:\cU\to \cV$ between rigid bar compatible $\cU$ and $\cV$ is itself called \emph{rigid bar compatible} if it commutes with the above relative tensor products 
  $$F(B)\otimes_{F(A)} F(C)\isoto F(B\otimes_A C).$$
  Denote by $\Catrigbar\subseteq \Catsm_2$ the locally full subcategory of rigid bar compatible categories and functors.
  Denote by $\Mapo^\rigbar$ the space of rigid bar compatible functors.
\end{definition}

\begin{remark}
  The author does not know of an example where condition (\ref{i:relative-tensor-of-rigid-exist}) in \cref{d:rigid-bar-compatible} holds but condition (\ref{i:relative-tensor-of-rigid-is-rigid}) does not.
\end{remark}

\begin{lemma}\label{l:rigbar-lands-in-rex}
  The functor $\CAlgrig\colon \Cat_2^\otimes\to \Cat_1$ restricts to $\CAlgrig\colon \Catrigbar\to \Cat_1^\rex$.
\end{lemma}

\begin{proof}
  If $\cU\in \Catrigbar$, then $\CAlgrig(\cU)$ has an initial object $\one$ and all pushouts by \cref{p:bar-construction-is-pushout}.
  A rigid bar compatible functor $\cU\to \cV$ induces a right exact functor $\CAlgrig(\cU)\to \CAlgrig(\cV)$ by definition. 
\end{proof}

Let $\cC\in \Cat_1^\rex$. 
To prove that $\coSpan(\cC)$ is rigid bar compatible, we will need the following lemma:

\begin{lemma}\label{l:right-way-contractible-colimits}
  The right-way functor $\cC\to \coSpan(\cC)$ sends weakly contractible colimits in $\cC$ to 2-colimits in $\coSpan(\cC)$.
\end{lemma}

\begin{proof}
  Let $D\colon I\to \cC$ have a colimit $X\in \cC$ where $I$ is weakly contractible.
  We need to show that for every $Y\in \coSpan(\cC)$, the map
  $$\hom_{\coSpan(\cC)}(X,Y)\to\lim_{i\in I^\op}\hom_{\coSpan(\cC)}(D(i),Y)$$
  is an equivalence of categories.
  The hom-category in $\coSpan(\cC)$ is given by the slice category 
  $$\hom_{\coSpan(\cC)}(X,Y)\simeq \cC_{X\sqcup Y/}\simeq \cC_{X/}\times_{\cC} \cC_{Y/},$$ 
  so we need to show that
  $$\cC_{X/}\times_{\cC} \cC_{Y/}\to 
  \lim_{i\in I^\op}(\cC_{D(i)/}\times_{\cC} \cC_{Y/})\simeq 
  \lim_{i\in I^\op}\cC_{D(i)/}\times_{\lim_{I^\op}\cC} \lim_{I^\op}\cC_{Y/}$$
  is an equivalence of categories.
  As $I$ is weakly contractible, we get equivalences for the constant limits $\cC\isoto\lim_{I^\op}\cC$ and $\cC_{Y/}\isoto\lim_{I^\op}\cC_{Y/}$.
  
  It remains to show that $\cC_{X/}\to \lim_{i\in I^\op} \cC_{D(i)/}$ is an equivalence of categories.
  For this, it suffices to show that the functor $\cC^\op\to \Cat_1$ given by $Z\mapsto \cC_{Z/}$ preserves weakly contractible limits.
  We factor this functor as the composition
  \begin{equation*}
    \cC^\op \hookrightarrow \Fun(\cC,\mathrm{Spc})  \isoto (\Cat_1)^{\mathrm{lfib}}_{/\cC}\to (\Cat_1)_{/\cC}\to \Cat_1
  \end{equation*}
  and verify that each factor preserves weakly contractible limits:
  \begin{enumerate}
    \item The Yoneda embedding, $Z\mapsto \Map_\cC(Z,-)$, preserves all limits.
    \item Unstraightening is an equivalence of categories, and sends $\Map_\cC(Z,-)$ to the left fibration $\cC_{Z/}\to \cC$.
    \item The functor forgetting that the map to $\cC$ is a left fibration preserves all limits by \cite[Theorem 3.1.5.1]{htt}\footnote{The cited result is for cocartesian fibrations, but this implies the result for left fibrations.}.
    \item The functor forgetting the map to $\cC$ preserves weakly contractible limits by \cite[Lemma 2.2.7]{gepner2021operadsasanalyticmonads}.
  \end{enumerate}
\end{proof}

\begin{lemma} \label{l:cospan-is-rigbar}
  The functor $\coSpan\colon \Cat_1^\rex\to \Cat_2^\otimes$ factors through $\Catrigbar$. 
\end{lemma}

\begin{proof}
  Let $\cC\in \Cat_1^\rex$, we need to show that $\coSpan(\cC)$ is rigid bar compatible.
  Consider $A\in \CAlgrig(\coSpan(\cC))$ and $B,C\in \CAlgrig(\coSpan(\cC))_{A/}$.
  Consider also $X,Y\in \coSpan(\cC)$.
  The bar complex $\Br_A^{\coSpan(\cC)}(X\sqcup B,C\sqcup Y)$ comes from a diagram of commutative algebras where all entries are rigid, hence $\Br_A^{\coSpan(\cC)}(X\sqcup B,C\sqcup Y)$ is the image of $\Br_A^{\cC}(X\sqcup B,C\sqcup Y)$ under the right-way functor $\cC\to \coSpan(\cC)$ by \cref{p:can-equiv-rig}.

  Observe that $B$ and $C$ are bar compatible over $A$ in $\cC$:
  the colimit of $\Br_A^{\cC}(X\sqcup B,C\sqcup Y)$ exists, as it is the pushout $X\sqcup B\sqcup_A C\sqcup Y$, and there is an isomorphism of pushouts  
  $$(X\sqcup B)\sqcup_A (C\sqcup Y)\isoto X\sqcup (B\sqcup_A C)\sqcup Y.$$
  It follows that $B$ and $\cC$ are also bar compatible in $\coSpan(\cC)$, as the right-way functor $\cC\to \coSpan(\cC)$ is symmetric monoidal and sends simplicial colimits to 2-colimits (\cref{l:right-way-contractible-colimits}).
  Moreover, the commutative algebra structure on $B\sqcup_A C$ is induced from the right-way functor, so it is rigid by \cref{l:can-is-rig}.

  Finally, a right exact functor $F:\cC\to \cD$ commutes with pushouts, so $F\colon \coSpan(\cC)\to \coSpan(\cD)$ is rigid bar compatible. 
\end{proof}

To demonstrate that rigid bar compatibility is not an entirely ad-hoc notion, we give a more natural example of a 2-category which is rigid bar compatible.

\begin{example}\label{e:prl-is-rigid-bar-compatible}
  $\Prl_\cB$ is rigid bar compatible for every $\cB\in \CAlg(\Prl)$.
  Indeed, $\Prl_\cB$ has all geometric realizations and they commute with the tensor product in each variable, so all modules are bar compatible \cite[Example 4.4.2.11]{ha}.
  The relative tensor product of rigid commutative algebras in $\Prl_\cB$ is itself rigid by \cite[Proposition 4.53]{ramzi2024locallyrigidinftycategories} (also follows from \cite[Proposition C.6.4]{arinkin2022stacklocalsystemsrestricted}).
\end{example}

The last remaining ingredient is that, in the context of rigid bar compatible 2-categories, the forgetful functor satisfies cobase change.

\begin{lemma} \label{l:forgetful-cobase-change}
  Suppose $\cU\in \Catrigbar$, then the forgetful functor $U:\CAlgrig(\cU)\to \cU$ satisfies cobase change. 
\end{lemma}

\begin{proof}
  By \cref{p:bar-construction-is-pushout}, the category $\CAlgrig(\cU)$ has all pushouts, and they are given by
  \[\begin{tikzcd}
    A & B \\
    C & {B\otimes_A C.}
    \arrow[from=1-1, to=1-2]
    \arrow[from=1-1, to=2-1]
    \arrow[from=1-2, to=2-2]
    \arrow[from=2-1, to=2-2]
    \arrow["\lrcorner"{anchor=center, pos=0.125, rotate=180}, draw=none, from=2-2, to=1-1]
  \end{tikzcd}\]
  Moreover, \cref{c:com-rigid-map-adjoint} shows that $A\to B$ is left adjoint in $\Mod_A(\cU)$.
  Therefore, by \cref{l:beck-chevalley-bar-compatible}, the image of this square in $\cU$ has an invertible Beck-Chevalley 2-morphism.
\end{proof}

\subsection{The adjunction}
\Cref{l:cospan-is-rigbar,l:rigbar-lands-in-rex} give us functors with compatible source and target:
\begin{align*}
  \coSpan&\colon  \Cat_1^\rex\to \Catrigbar\\
  \CAlgrig&\colon  \Catrigbar\to \Cat_1^\rex.
\end{align*} 
We now prove that they are adjoint.

\begin{theorem} \label{t:cospan_cfrob_adj}
  There is a coreflective adjunction
  $$\coSpan\colon 
  \Cat_1^\rex\rightleftarrows\Catrigbar
  \colon \CAlgrig,$$
  meaning that for every $\cC\in \Cat_1^\rex$ and $\cU\in \Catrigbar$ there is a natural isomorphism
  $$\Mapo^\rigbar(\coSpan(\cC),\cU)\simeq \Map^\rex(\cC,\CAlgrig(\cU)).$$
\end{theorem}

\begin{proof}
  The unit map is the isomorphism $u\colon \cC\xrightarrow{\sim} \CAlgrig(\coSpan(\cC))$ of \cref{p:can-equiv-rig}.
  The counit map $c\colon \coSpan(\CAlgrig(\cU))\to \cU$ is constructed from the universal property of cospans (\cref{p:universal-propertey-of-cospans}) by the forgetful functor $U\colon \CAlgrig(\cU)\to \cU$, where cobase change holds by \cref{l:forgetful-cobase-change}.
  It remains to verify the zigzag identities.
  
  For the first identity, the following diagram commutes:
  \[\begin{tikzcd}
    \cC\arrow[r,"u"]\arrow[d] &\CAlgrig(\coSpan(\cC))\arrow[d]\arrow[rd,"U"] &\\
    \coSpan(\cC)\arrow[r,"u"]& \coSpan(\CAlgrig(\coSpan(\cC)))\arrow[r,"c"] & \coSpan(\cC)
  \end{tikzcd}\]
  where the vertical maps are the right-way functors.
  By the universal property of cospans, to show that the bottom composes to the identity it is enough to show that the top composes to the right-way functor.
  And indeed, the unit is given by a lifting of the right-way functor.

  For the second zigzag identity, the following diagram commutes:
  \[\begin{tikzcd}
    {\CAlgrig(\cU)} & {\CAlgrig(\coSpan(\CAlgrig(\cU)))} & {\CAlgrig(\cU)} \\
    & {\coSpan(\CAlgrig(\cU))} & \cU
    \arrow["u", from=1-1, to=1-2]
    \arrow[from=1-1, to=2-2]
    \arrow["c", from=1-2, to=1-3]
    \arrow["U", from=1-2, to=2-2]
    \arrow["U", from=1-3, to=2-3]
    \arrow["c", from=2-2, to=2-3]
  \end{tikzcd}\]
  where the sloped arrow is the right-way functor.
  The counit came from the forgetful functor via the universal property, so the composition along the bottom is the forgetful functor $U:\CAlgrig(\cU)\to \cU$.
  Hence, the composition along the top is a coproduct preserving functor $f\colon \CAlgrig(\cU)\to \CAlgrig(\cU)$ with a commuting diagram 
  \[\begin{tikzcd}
    {\CAlgrig(\cU)} && {\CAlgrig(\cU)} \\
    & \cU.
    \arrow["f", from=1-1, to=1-3]
    \arrow["U"', from=1-1, to=2-2]
    \arrow["U", from=1-3, to=2-2]
  \end{tikzcd}\]
  We claim the space of such functors is contractible, and in particular our functor is isomorphic to the identity.
  To prove this claim, note that $U$ is symmetric monoidal with respect to the cocartesian structure on $\CAlgrig(\cU)$, so the diagram lifts uniquely to 
  \[\begin{tikzcd}
    {\CAlgrig(\cU)} && {\CAlgrig(\cU)} \\
    & {\CAlg(\cU),}
    \arrow["f", from=1-1, to=1-3]
    \arrow[hook, from=1-1, to=2-2]
    \arrow[hook', from=1-3, to=2-2]
  \end{tikzcd}\]
  where $f$ is unique as $\CAlgrig(\cU)\hookrightarrow \CAlg(\cU)$ is fully faithful.
\end{proof}

\begin{remark}
  The fact that the adjunction in \cref{t:cospan_cfrob_adj} is coreflective implies that the functor $\coSpan\colon \Cat_1^\rex\to \Catrigbar$ is fully faithful.
  In fact, it follows that $\coSpan\colon \Cat_1^\rex\to \Catsm_2$ is also fully faithful.
  This can also be deduced directly from the universal property of cospans.
\end{remark}

\bibliographystyle{alpha}
\bibliography{references}
\end{document}